\documentclass[10pt, a4paper, twocolumn ,titlepage, egregdoesnotlikesansseriftitles]{scrartcl}
\pdfoutput=1
\usepackage{geometry}
\usepackage{mathtools} 
\usepackage{stmaryrd} 
\usepackage{amsmath} 
\usepackage{amsthm} 
\usepackage{amssymb} 
\usepackage[english]{babel} 
\usepackage[hidelinks]{hyperref}
\usepackage[all]{xy} 
\usepackage{graphicx} 
\usepackage{chngpage} 
\usepackage{indentfirst} 
\usepackage[T1]{fontenc}
\usepackage[latin9]{inputenc}
\usepackage{prettyref} 
\usepackage{xcolor}
\usepackage[font={small,it}]{caption} 
\usepackage{sectsty} 

\setkomafont{section}{\large}
\graphicspath{{figures/}}

\numberwithin{equation}{section}
\numberwithin{figure}{section}

\theoremstyle{plain}
\newtheorem{thm}{Theorem}
\newtheorem{lem}[thm]{Lemma}
\newtheorem{prop}[thm]{Proposition}
\newtheorem{cor}[thm]{Corollary}

\theoremstyle{definition}
\newtheorem*{defn*}{Definition}
\newtheorem*{example*}{Example}

\newcommand{\orcid}[1]{\href{https://orcid.org/#1}{\includegraphics[width=10pt]{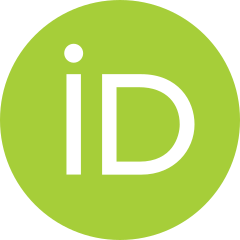}  orcid.org/#1}}

\title{\vspace{-0.5cm}\LARGE \textsc{\scalebox{0.92}[1.0]{Applications of Structural Statistics:}\\\scalebox{0.92}[1.0]{Geometric Inference in Exponential Families}}\vspace{-0.5cm}}
\author{\large P. Michl \orcid{0000-0002-6398-0654}}
\date{}

\begin{document}

\twocolumn[ 
\begin{@twocolumnfalse}  

\maketitle 

\begin{abstract}
\vspace{-1.9cm}
\begin{adjustwidth}{10mm}{10mm}
Exponential families comprise a broad class of statistical models and parametric families like normal distributions, binomial distributions, gamma distributions or exponential distributions. Thereby the formal representation of its probability distributions induces a confined intrinsic structure, which appears to be that of a dually flat statistical manifold. Conversely it can be shown, that any dually flat statistical manifold, which is given by a regular Bregman divergence uniquely induced a regular exponential family, such that exponential families may - with some restrictions - be regarded as a universal representation of dually flat statistical manifolds. This article reviews the pioneering work of Shun'ichi Amari about the intrinsic structure of exponential families in terms of structural stratistics.
\end{adjustwidth}
\vspace{0.5cm}
\end{abstract}
\end{@twocolumnfalse}
] 

\section{Introduction}

In accordance to it's historical roots classical statistical theory has been formulated to describe repeatable experiments in terms of random variables. A drawback that accompanies this language is the difficulty to integrate and describe abstract structural knowledge. Rising theories like deep learning and complex networks dynamics, however, impressively demonstrate, that statistical modeling and inference can greatly benefit from the integration of abtract structural assumptions and in particular in the domains of complex natural data.

It is therefore not surprising that this development led to a growing interest in alternative approaches to formulate statistical theory. In particular \textsc{S. Amari} pursued a fundamentally different approach by focusing on the embedding function space of the probability distributions \cite{Amari1987}. This view motivated the reformulation of statistical theory by means of structural statistics \cite{Michl2019, Michl2020}. An important application and showcase are exponential families, which can be completely characterized their geometric structure in terms of dually flat statistical manifolds.

\section{Primary affine structure of Exponential Families}
\begin{defn*}[Exponential family]
\label{def:Exponential-family} \emph{Let $(\Omega,\,\mathcal{F_{\theta}})$
be a statistical model over a measurable space $(\Omega,\,\Sigma_{\Omega})$.
Then $(\Omega,\,\mathcal{F})$ is termed an exponential family if
and only if there exists an invertible function 
\[
\xi:\mathrm{dom}\theta\rightarrow\mathbb{R}^{n}
\]
a sufficient statistic 
\[
T:(\Omega,\,\Sigma_{\Omega})\rightarrow(\mathbb{R}^{n},\,\mathcal{B}(\mathbb{R}^{n}))
\]
 and a scalar function 
\[
f:\mathrm{dom}\theta\rightarrow\mathbb{R}
\]
such that for any$P_{\theta}\in\mathcal{F}$ and $\sigma\in\Sigma_{\Omega}$
it holds that:
\begin{align}
P_{\theta}[\sigma]= & \int_{\sigma}\exp\left(\boldsymbol{\xi}(\theta_{P})\cdot\boldsymbol{T}(\omega)\right)\mathrm{d}\mu(\omega)\label{eq:expfamily:standardform}\\
 & \cdot\int_{\sigma}\exp f(\theta_{P})\mathrm{d}\mu(\omega)\nonumber 
\end{align}
}
\end{defn*}
Since $T$ is a sufficient statistic a Markov morphism, given by $\mathcal{T}(P)[\sigma]=P[T^{-1}(\sigma)]$
is globally invertible and its restriction to $(\Omega,\,\mathcal{F})$
yields a statistical isomorphism $\mathfrak{T}\in\mathrm{iso(\mathbf{Stat})}$
to a statistical model $(X,\,\mathcal{P})\coloneqq\mathrm{img}\,\mathfrak{T}$
over a measurable space $(X,\,\Sigma_{X})$. Then $\eta\coloneqq\theta\circ\xi^{-1}$
is an identifiable parametrisation of $(X,\,\mathcal{P})$ and by
the definition
\[
\psi\coloneqq\xi^{-1}\circ f\circ\xi
\]
it follows that:
\begin{equation}
P_{\eta}[\sigma]=\int_{\sigma}\exp\left(\boldsymbol{\eta}_{P}\cdot\boldsymbol{x}-\psi(\boldsymbol{\eta}_{P})\right)\mathrm{d}\boldsymbol{x},\,\forall\sigma\in\Sigma_{X}\label{eq:expfamily:canonicalform}
\end{equation}
Without loss of generality any exponential family, as defined in \ref{eq:expfamily:standardform},
may therefore be assumed to be given by probability distributions
with a representation \ref{eq:expfamily:canonicalform}. This representation
is termed the canonical form of an exponential family and the parametrisation
$\eta$ the canonical, or natural parametrisation.

\begin{defn*}[Natural parametrisation]
\label{def:Natural-parametrisation} \emph{Let $(X,\,\mathcal{P})\in\mathrm{ob}(\mathbf{Stat})$
be an exponential family in canonical form, then the corresponding
canonical parametrisation $\eta$ is termed a natural parametrisation
of $(X,\,\mathcal{P})$ and the parameter vectors $\boldsymbol{\eta}_{P}\in\mathrm{dom}\,\eta\subseteq\mathbb{R}^{n}$
are termed natural parameters. }\textbf{Remark}\emph{: Exponential
families, given by the notation $(X,\,\mathcal{P}_{\eta})$ implicate
a canonical form and a natural parametrisation $\eta$.}
\end{defn*}
The function $\psi:\mathbb{R}^{n}\to\mathbb{R}$, given by an exponential
family in canonical form is known as the cumulant generating function
and may be regarded as a normalisation factor, that implements the
normalisation condition of the probability distribution:
\begin{equation}
\int_{X}\exp\left(\boldsymbol{\eta}_{P}\cdot\boldsymbol{x}-\psi(\boldsymbol{\eta}_{P})\right)\mathrm{d}\boldsymbol{x}=1\label{eq:expfam:normalization-condition}
\end{equation}
Since $\psi(\boldsymbol{\eta}_{P})$ is independent of $\boldsymbol{x}$
it may be pulled out of the integral and a rearrangement of equation
\ref{eq:expfam:normalization-condition} yields: 
\begin{equation}
\psi(\boldsymbol{\eta}_{P})=\log\int_{X}\exp(\boldsymbol{\eta}_{P}\cdot\boldsymbol{x})\mathrm{d}\boldsymbol{x}\label{eq:expfam:func}
\end{equation}
Due to this dependency the cumulant generating function relates different
statistical properties.
\begin{lem}
\label{lem:4.1}Let $\psi$ be the cumulant generating function of
an exponential family \emph{$(X,\,\mathcal{P}_{\eta})$}, then $\psi$
is convex with respect to $\eta$ and it's first and second order
derivatives are given by:
\begin{eqnarray}
\nabla_{\eta}\psi(\boldsymbol{\eta}_{P}) & = & \mathrm{E}_{\eta}[\boldsymbol{x}]\label{eq:expfam:nablaf}\\
\nabla_{\eta}^{2}\psi(\boldsymbol{\eta}_{P}) & = & \mathrm{Var_{\eta}}[\boldsymbol{x}]\label{eq:expfam:hessianf}
\end{eqnarray}
where $\mathrm{E}_{\eta}[\boldsymbol{x}]\in\mathbb{R}^{n}$ and $\mathrm{Var_{\eta}}[\boldsymbol{x}]\in\mathbb{R}^{n}$
respectively denote the vectorial expectation and variance of $\boldsymbol{x}$
with respect to $P_{\eta}$.
\end{lem}

\begin{proof}
Let $P_{\eta}\in\mathcal{P}$ and let $p_{\eta}$ be the density function
of $P_{\eta}$ over $(X,\,\Sigma_{X})$. Then the normalization condition
is:
\begin{equation}
\int_{X}p_{\eta}(\boldsymbol{x})\mathrm{d}\boldsymbol{x}=1\label{eq:expfam:normalization}
\end{equation}
The partial derivation $\partial_{i}$ of equation $\ref{eq:expfam:normalization}$
to the natural parameter $\eta_{i}$ yields:
\begin{eqnarray}
 &  & \partial_{i}\int_{X}p_{\eta}(\boldsymbol{x})\mathrm{d}\boldsymbol{x}\nonumber \\
 & = & \int_{X}\mathrm{\partial_{i}exp}\left(\sum_{i=1}^{n}\eta_{i}x_{i}-\psi(\boldsymbol{\eta}_{P})\right)\mathrm{d}\boldsymbol{x}\nonumber \\
 & = & \int_{X}(x_{i}-\partial_{i}\psi(\boldsymbol{\eta}_{P}))p_{\eta}(\boldsymbol{x})\mathrm{d}\boldsymbol{x}\stackrel{\ref{eq:expfam:normalization}}{=}0\label{eq:expfam:partialf}
\end{eqnarray}
Therefore:
\begin{eqnarray*}
 &  & \nabla_{\eta}\psi(\boldsymbol{\eta}_{P})\\
 & \stackrel{\ref{eq:expfam:normalization}}{=} & \nabla_{\eta}\psi(\boldsymbol{\eta}_{P})\int_{X}p_{\eta}(\boldsymbol{x})\mathrm{d}\boldsymbol{x}\\
 & = & \int_{X}\nabla_{\eta}\psi(\boldsymbol{\eta}_{P})p_{\eta}(\boldsymbol{x})\mathrm{d}\boldsymbol{x}\\
 & \stackrel{\ref{eq:expfam:partialf}}{=} & \int_{X}\boldsymbol{x}p_{\eta}(\boldsymbol{x})\mathrm{d}\boldsymbol{x}\\
 & = & \mathrm{E_{\eta}}[\boldsymbol{x}]
\end{eqnarray*}
This proves equation \ref{eq:expfam:nablaf}. A further partial derivation
of \ref{eq:expfam:partialf} with respect to the natural parameter
$\eta_{j}$ yields: 
\begin{eqnarray}
0 & \stackrel{\ref{eq:expfam:partialf}}{=} & \partial_{j}\int_{X}(x_{i}-\partial_{i}\psi(\boldsymbol{\eta}_{P}))p_{\eta}(\boldsymbol{x})\mathrm{d}\boldsymbol{x}\nonumber \\
 & = & \int_{X}(x_{i}-\partial_{i}\psi(\boldsymbol{\eta}_{P}))\partial_{j}p_{\eta}(\boldsymbol{x})\mathrm{d}\boldsymbol{x}\nonumber \\
 &  & -\int_{X}\partial_{j}\partial_{i}\psi(\boldsymbol{\eta}_{P})\int_{X}p_{\eta}(\boldsymbol{x})\mathrm{d}\boldsymbol{x}\nonumber \\
 & \stackrel{\ref{eq:expfam:partialf}}{=} & \int_{X}(x_{i}-\partial_{i}\psi(\boldsymbol{\eta}_{P}))(x_{j}-\partial_{j}\psi(\boldsymbol{\eta}))p_{\eta}(\boldsymbol{x})\mathrm{d}\boldsymbol{x}\label{eq:expfam:dpartialf}\\
 &  & -\partial_{j}\partial_{i}\psi(\boldsymbol{\eta}_{P})\nonumber 
\end{eqnarray}
Therefore:
\begin{eqnarray*}
 &  & \nabla_{\eta}^{2}\psi(\boldsymbol{\eta}_{P})\\
 & \stackrel{\ref{eq:expfam:dpartialf}}{=} & \int_{X}(\boldsymbol{x}-\nabla_{\eta}\psi(\boldsymbol{\eta}_{P}))(\boldsymbol{x}-\nabla_{\eta}\psi(\boldsymbol{\eta}_{P}))p_{\eta}(\boldsymbol{x})\mathrm{d}\boldsymbol{x}\\
 & \stackrel{\ref{eq:expfam:nablaf}}{=} & \int_{X}(\boldsymbol{x}-\mathrm{E}[\boldsymbol{x}])^{2}p_{\eta}(\boldsymbol{x})\mathrm{d}\boldsymbol{x}\\
 & = & \mathrm{Var_{\eta}}[x]
\end{eqnarray*}
This proves equation \ref{eq:expfam:hessianf}. Furthermore since:
\[
\nabla_{\eta}^{2}\psi(\boldsymbol{\eta}_{p})=\mathrm{Var_{\eta}}[x]\geq0,\,\forall\boldsymbol{\eta}_{p}\in\mathrm{img}\eta
\]
\end{proof}
The Hessian matrix of $\psi$ is positive definite and therefore $\psi$
is convex. The convexity of $\psi$ may therefore be used to induce
a Riemannian metric by the Bregman divergence.
\begin{lem}
\label{lem:4.2}Let $\psi$ be the cumulant generating function of
an exponential family \emph{$(X,\,\mathcal{P}_{\eta})$.} Then the
Bregman divergence given by $D_{\psi}$, is the dual Kullback-Leibler
divergence, such that:
\begin{equation}
D_{\psi}[P\parallel Q]=D_{\mathrm{KL}}^{*}[P\parallel Q],\,\forall P,\,Q\in\mathcal{P}\label{eq:lem:3.14:1}
\end{equation}
\end{lem}

\begin{proof}
Let $P,\,Q\in\mathcal{P}$ and $p,\,q$ their respectively density
functions over $(X,\,\Sigma_{X})$. By calculating:
\begin{equation}
\log p_{\eta}(\boldsymbol{x})=\boldsymbol{\eta}_{p}\cdot\boldsymbol{x}-\psi(\boldsymbol{\eta}_{p})\label{eq:expfam:logp}
\end{equation}
the Bregman divergence is derived by:
\begin{eqnarray*}
 &  & D_{\psi}[P\parallel Q]\\
 & \stackrel{}{=} & \psi(\boldsymbol{\eta}_{P})-\psi(\boldsymbol{\eta}_{Q})\\
 &  & -\nabla_{\eta}\psi(\boldsymbol{\eta}_{Q})\cdot(\boldsymbol{\eta}_{Q}-\boldsymbol{\eta}_{P})\\
 & = & \left(\int_{X}(\boldsymbol{\eta}_{Q}\cdot\boldsymbol{x})q(\boldsymbol{x})\mathrm{d}\boldsymbol{x}-\psi(\boldsymbol{\eta}_{Q})\right)\\
 &  & -\left(\int_{X}(\boldsymbol{\eta}_{P}\cdot\boldsymbol{x})q(\boldsymbol{x})\mathrm{d}\boldsymbol{x}-\psi(\boldsymbol{\eta}_{P})\right)\\
 & \stackrel{\ref{eq:expfam:logp}}{=} & \int_{X}q_{\eta}(\boldsymbol{x})\log q(\boldsymbol{x})\mathrm{d}\boldsymbol{x}\\
 &  & -\int_{X}q(\boldsymbol{x})\log p(\boldsymbol{x})\mathrm{d}\boldsymbol{x}\\
 & = & \int_{X}q(\boldsymbol{x})\log\frac{q(\boldsymbol{x})}{p(\boldsymbol{x})}\mathrm{d}\boldsymbol{x}\\
 & \stackrel{}{=} & D_{\mathrm{KL}}^{*}[P\parallel Q]
\end{eqnarray*}
\end{proof}
\begin{lem}
\label{lem:4.3}Let $\psi$ be the cumulant generating function of
an exponential family \emph{$(X,\,\mathcal{P}_{\eta})$.} Then the
Bregman divergence given by $D_{\psi}$, induces a Riemannian metric,
which is given by the Fisher information: 
\begin{equation}
g_{P,\psi}=\mathrm{I}[P\mid X]\label{eq:expfam:metric}
\end{equation}
\end{lem}

\begin{proof}
Let $P\in\mathcal{P}$ and $p$ the probability density function of
$P$ over $(X,\,\Sigma_{X})$, then:
\begin{equation}
\nabla_{\eta}\log p_{\eta}(\boldsymbol{x})=\nabla_{\eta}(\boldsymbol{\eta}_{P}\cdot\boldsymbol{x}-\psi(\boldsymbol{\eta}_{P}))=\boldsymbol{x}-\nabla_{\eta}\psi(\boldsymbol{\eta}_{P})\label{eq:expfam:dellogp}
\end{equation}
The Riemannian metric, which is induced by the Bregman divergence
$D_{\psi}$ is therefore given by:

\begin{eqnarray*}
g_{P,\psi} & \stackrel{}{=} & \nabla_{\eta}^{2}\psi(\boldsymbol{\eta}_{P})\\
 & \stackrel{\ref{eq:expfam:dpartialf}}{=} & \int_{X}(\boldsymbol{x}-\nabla_{\eta}\psi(\boldsymbol{\eta}_{P}))^{2}p_{\eta}(\boldsymbol{x})\mathrm{d}\boldsymbol{x}\\
 & \stackrel{\ref{eq:expfam:dellogp}}{=} & \int_{X}(\nabla_{\eta}\log p_{\eta}(\boldsymbol{x}))^{2}p_{\eta}(\boldsymbol{x})\mathrm{d}\boldsymbol{x}\\
 & \stackrel{\mathrm{def}}{=} & \mathrm{I}[P_{\eta}\mid X]
\end{eqnarray*}
\end{proof}
\begin{lem}
\label{lem:4.4}Let $\psi$ be the cumulant generating function of
an exponential family $(X,\,\mathcal{P}_{\eta})$ and $(X,\,\mathcal{P}_{\eta},\,D_{\psi})$
the Riemannian statistical manifold, given by the Bregman divergence
$D_{\psi}$. Then the $\eta$-affine geodesics in $(X,\,\mathcal{P}_{\eta},\,D_{\psi})$
are given by exponential families.
\end{lem}

\begin{proof}
The $\eta$-affine geodesics in $(X,\,\mathcal{P_{\eta}},\,D_{\psi})$
are given by affine linear curves in the $\eta$-parametrisation.
For $P,\,Q\in\mathcal{P}$ and $\boldsymbol{\eta}_{P,Q}(t)=(1-t)\boldsymbol{\eta}_{P}+t\boldsymbol{\eta}_{Q}$,
with $t\in[0,\,1]$ let $\gamma_{P,Q}(t)$ be the $\eta$-affine geodesic
connecting $P$ with $Q$ and $\mathrm{d}_{X}\gamma_{P,Q}(t)$ the
probability density function of $\gamma_{P,Q}(t)$ over $(X,\,\Sigma_{X})$.
Then then representation of $\mathrm{d}_{X}\gamma_{P,Q}(t)$ in the
$\eta$-parametrisation is given by:
\[
\mathrm{d}_{X}\gamma_{P,Q}(t)(\boldsymbol{x})=\exp\left(t(\boldsymbol{\eta}_{Q}-\boldsymbol{\eta}_{P})\cdot\boldsymbol{x}+\boldsymbol{\eta}_{P}\cdot\boldsymbol{x}-\psi_{\eta}(t)\right)
\]
Let $p,\,q$ be the respective probability densities of $P$ and $Q$
over $(X,\,\Sigma_{X})$. A $\log$-transformation and a subsequent
substitution of $\boldsymbol{\eta}_{P}$ and $\boldsymbol{\eta}_{Q}$
by equation \ref{eq:expfam:logp} yields:
\[
\log\mathrm{d}_{X}\gamma_{P,Q}(t)(\boldsymbol{x})=(1-t)\log p(\boldsymbol{x})+t\log q(\boldsymbol{x})-\psi(t)
\]
A further $\exp$-transformation and subsequent integration over a
measurable set $\sigma\in\Sigma$ gives:

\begin{align*}
\gamma_{P,Q}(t)[\sigma] & =\int_{\sigma}\exp\left((1-t)\log p(\boldsymbol{x})\right)\mathrm{d}\boldsymbol{x}\\
 & \cdot\int_{\sigma}\exp\left(t\log q(\boldsymbol{x})-\psi(t)\right)\mathrm{d}\boldsymbol{x}
\end{align*}
For varying $P,\,Q\in\mathcal{P}$ this yields a generic representation
of $\eta$-affine geodesics in $(X,\,\mathcal{P_{\eta}},\,D_{\psi})$,
by are exponential families with respect to the curve parameter $t$.
\end{proof}
\begin{lem}
\label{lem:4.5}Let $\psi$ be the cumulant generating function of
an exponential family $(X,\,\mathcal{P}_{\eta})$ and $(X,\,\mathcal{P}_{\eta},\,D_{\psi})$
the Riemannian statistical manifold, given by the Bregman divergence
$D_{\psi}$. Then the $\eta$-affine geodesics in $(X,\,\mathcal{P}_{\eta},\,D_{\psi})$
are flat with respect to the Fisher information.
\end{lem}

\section{Dual affine structure of Exponential Families}

In the purpose, to study the structure, obtained by a Legendre transformation
a further parametric family is introduced, that is shown to cover
this dual structure. This parametric family is the mixture family. 
\begin{defn*}[Mixture family]
\label{def:Mixture-family} \emph{Let 
\[
(\Omega,\,\mathcal{F_{\theta}})\in\mathrm{ob}(\mathbf{Stat})
\]
be a statistical model over a measurable space $(\Omega,\,\Sigma_{\Omega})$.
Then $(\Omega,\,\mathcal{F})$ is a mixture family if and only if
there is an invertible function $w:\mathrm{dom}\theta\rightarrow\mathbb{R}^{n+1}$
with $w_{i}(\theta_{p})>0,\,\forall i$ and $\sum_{i=0}^{n}w_{i}(\theta_{p})=1,\,\forall\theta_{p}$
and $n+1$ pairwise independent probability distributions $Q_{i}$
over $(\Omega,\,\Sigma_{\Omega})$, such that for all $P_{\theta}\in\mathcal{F_{\theta}}$:
\begin{equation}
P_{\theta}[\sigma]=\sum_{i=0}^{n}w_{i}(\theta_{p})Q_{i}[\sigma],\,\forall\sigma\in\Sigma_{\Omega}\label{eq:mixfam:stanform}
\end{equation}
}
\end{defn*}
The normalization condition can be used to restrict the number of
parameters, by the definition:
\begin{equation}
R(\theta_{p})[\sigma]=\left(1-\sum_{i=1}^{n}w_{i}(\theta_{p})\right)Q_{0}[\sigma]\label{eq:mixfam:rest}
\end{equation}
Therefore:\emph{
\[
P_{\theta}[\sigma]=\sum_{i=1}^{n}w_{i}(\theta_{p})Q_{i}[\sigma]+R(\theta_{p})[\sigma]
\]
}This transformation is a globally invertible Markov morphism, and
its restriction to $(\Omega,\,\mathcal{F})$ yields a statistical
isomorphism $\mathfrak{T}\in\mathrm{iso(\mathbf{Stat})}$ to a statistical
model $(X,\,\mathcal{P})\coloneqq\mathrm{img}\mathcal{\mathfrak{T}}$.
Then $\mu\coloneqq w\circ\theta^{-1}$ is an identifiable parametrisation
of $(X,\,\mathcal{P})$ and since $w$ defines a probability distribution
over $\theta$ it may be regarded as the expected value $\mathrm{E_{\theta}}[\boldsymbol{x}]$
with respect to $P_{\theta}$. By the definition of $\varphi\coloneqq-w^{-1}\circ R\circ w$
it follows that:\emph{
\begin{equation}
P_{\mu}[\sigma]=\boldsymbol{\mu}_{P}\cdot\boldsymbol{Q}[\sigma]-\varphi(\boldsymbol{\mu}_{P})[\sigma]\label{eq:mixfam:canform}
\end{equation}
}Without loss of generality any mixture family may therefore be assumed
to be given by equations \ref{eq:mixfam:canform}. This representation
is termed the canonical form of a mixture family and the parametrisation
$\mu$ the expectation parametrisation.
\begin{defn*}[Expectation parametrisation]
\label{def:Expectation-parametrisation} \emph{Let 
\[
(X,\,\mathcal{P}_{\theta})\in\mathrm{ob}(\mathbf{Stat})
\]
Then a parametrisation $\mu$ with $\mathrm{dom}\mu\subseteq\mathbb{R}^{n}$,
which is given by $\boldsymbol{\mu}_{P}=\mathrm{E_{\theta}}[\boldsymbol{x}]$,
where $\mathrm{E_{\theta}}[\boldsymbol{x}]$ denotes the vectorial
expectation of $\boldsymbol{x}$ with respect to $P_{\theta}$, is
termed an expectation parametrisation with respect to $\theta$ and
parameter vectors $\boldsymbol{\mu}_{P}\in\mathrm{dom}\mu$ are termed
expectation parameters.}
\end{defn*}
\begin{lem}
\label{lem:4.6}Let $\psi$ be the cumulant generating function of
an exponential family \emph{$(X,\,\mathcal{P}_{\eta})$} and $\varphi$
the negative entropy $-\mathrm{H}[P_{\eta}\mid X]$. Then the dual
parametrisation $\eta{}^{*}$ is given by the expectation parametrisation
$\mu$ and the Legendre dual function $\psi^{*}$ by $\varphi$, such
that:
\begin{eqnarray}
\boldsymbol{\eta}_{P}^{*} & = & \boldsymbol{\mu}_{P}\label{eq:expfam:dualparam}\\
\psi^{*}(\boldsymbol{\eta}_{P}^{*}) & = & \varphi(\boldsymbol{\mu}_{P})\label{eq:expfam:dualfunc}
\end{eqnarray}
\end{lem}

\begin{proof}
Let $P_{\eta}\in\mathcal{P}$ and $p_{\eta}$ the density function
of $P_{\eta}$ over $(X,\,\Sigma_{X})$. Then from equation \ref{eq:expfam:nablaf}
is follows, that:
\[
\boldsymbol{\eta}_{P}^{*}=\nabla_{\eta}\psi(\boldsymbol{\eta}_{P})\stackrel{\ref{eq:expfam:nablaf}}{=}\mathrm{E_{\eta}}[\boldsymbol{x}]\stackrel{\mathrm{def}}{=}\boldsymbol{\mu}_{P}
\]
the dual parametrisation $\boldsymbol{\eta}_{P}^{*}$ directly follows
from . The Legendre dual function $\psi^{*}$ is derived by:
\begin{eqnarray*}
\psi^{*}(\boldsymbol{\eta}_{P}^{*}) & \stackrel{}{=} & \boldsymbol{\eta}_{P}\cdot\nabla_{\eta}\psi(\boldsymbol{\eta}_{P})-\psi(\boldsymbol{\eta}_{P})\\
 & \stackrel{\ref{eq:expfam:normalization}}{=} & \int_{X}(\boldsymbol{\eta}_{P}\cdot\boldsymbol{x})p_{\eta}(\boldsymbol{x})\mathrm{d}\boldsymbol{x}-\psi(\boldsymbol{\eta}_{P})\int_{X}p_{\eta}(\boldsymbol{x})\mathrm{d}\boldsymbol{x}\\
 & = & \int_{X}p_{\eta}(\boldsymbol{x})(\boldsymbol{\eta}_{P}\cdot\boldsymbol{x}-\psi(\boldsymbol{\eta}_{P}))\mathrm{d}\boldsymbol{x}\\
 & \stackrel{\ref{eq:expfam:func}}{=} & \int_{X}p_{\eta}(\boldsymbol{x})\log p_{\eta}(\boldsymbol{x})\mathrm{d}\boldsymbol{x}\\
 & = & -\mathrm{H}[P_{\eta}\mid X]
\end{eqnarray*}
\end{proof}
\begin{lem}
\label{lem:4.7}Let $\psi$ be the cumulant generating function of
an exponential family \emph{$(X,\,\mathcal{P}_{\eta})$}. Then the
The Bregman divergence, given by the Legendre dual function $\varphi$
is the Kullback-Leibler divergence:
\[
D_{\varphi}[P\parallel Q]=D_{\mathrm{KL}}[P\parallel Q]
\]
\end{lem}

\begin{proof}
From lemma \ref{lem:3.9}, equation \ref{eq:divergence_dual_divergence}
it follows, that:
\begin{eqnarray*}
D_{\varphi}[P\parallel Q] & \stackrel{\text{\ref{eq:divergence_dual_divergence}}}{=} & D_{\varphi^{*}}[Q_{\mu}^{*}\parallel P_{\mu}^{*}]\\
 & \stackrel{\mathrm{def}}{=} & D_{\psi}[Q_{\eta}\parallel P_{\eta}]\\
 & \stackrel{\text{\ref{eq:lem:3.14:1}}}{=} & D_{\mathrm{KL}}[P\parallel Q]
\end{eqnarray*}
\end{proof}
\begin{lem}
\label{lem:4.8}Let $\psi$ be the cumulant generating function of
an exponential family \emph{$(X,\,\mathcal{P}_{\eta})$}. Then the
Bregman divergence $D_{\varphi}$ of the Legendre dual function $\varphi=\psi^{*}$,
induces a Riemannian metric, which is given by the inverse Fisher
information: 
\begin{equation}
g_{P,\varphi}=\mathrm{I}[P\mid X]^{-1}\label{eq:expfam:dualmetric}
\end{equation}
\end{lem}

\begin{proof}
By applying theorem \ref{prop:3.2} the Bregman divergence $D_{\varphi}=D_{\psi^{*}}$
induces the dual Riemannian metric $g_{P,\psi}^{*}$, which by definition
is inverse to the Riemannian metric, induced by $D_{\psi}$, such
that $g_{P,\varphi}=g_{P,\psi}^{-1},\,\forall P\in\mathcal{P}$. From
lemma \ref{lem:4.3} it follows, that $g_{P,\varphi}=\mathrm{I}[P\mid X]^{-1}.$
\end{proof}
\begin{lem}
\label{lem:4.9}Let $\psi$ be the cumulant generating function of
an exponential family $(X,\,\mathcal{P}_{\eta})$ and $(X,\,\mathcal{P},\,D_{\varphi})$
the Riemannian statistical manifold, given by the Bregman divergence
$D_{\varphi}$of the Legendre dual function $\varphi=\psi^{*}$. Then
the $\mu$-affine geodesics in $(X,\,\mathcal{P}_{\mu},\,D_{\psi})$
are given by mixture families.
\end{lem}

\begin{proof}
$\mu$-affine geodesics in $(X,\,\mathcal{P}_{\mu},\,D_{\psi})$ are
given by affine linear curves in the $\mu$-parametrisation. For $P,\,Q\in\mathcal{P}$
and $\boldsymbol{\mu}_{P,Q}(t)=(1-t)\boldsymbol{\mu}_{P}+t\boldsymbol{\mu}_{Q}$,
with $t\in[0,\,1]$ let $\gamma_{P,Q}(t)$ be the $\mu$-affine geodesic
connecting $P$ with $Q$ and $\mathrm{d}_{x}\gamma_{P,Q}(t)$ the
probability density of $\gamma_{P,Q}(t)$ over $(X,\,\Sigma_{X})$,
then:
\[
\mathrm{d}_{X}\gamma_{P,Q}(t)(\boldsymbol{x})=(1-t){}_{\boldsymbol{\mu}}(\boldsymbol{x})+tq_{\boldsymbol{\mu}}(\boldsymbol{x})
\]
An integration over $\sigma$ yields:
\[
\gamma_{P,Q}(t)=\int_{\sigma}\left((1-t)p(\boldsymbol{x})+tq(\boldsymbol{x})\right)\mathrm{d}\boldsymbol{x}
\]
This is a mixture of probability distributions with respect to a mixing
parameter $t$. 
\end{proof}
\begin{lem}
\label{lem:4.10}Let $\psi$ be the cumulant generating function of
an exponential family $(X,\,\mathcal{P}_{\eta})$. Then the $\mu$-affine
geodesics are flat with regard to the Riemannian metric, induced by
the Bregman divergence $D_{\varphi}$of the Legendre dual function
$\varphi=\psi^{*}$.
\end{lem}

\section{Dually flat structure}
\begin{thm}[\emph{Structure of Exponential Families}]
\label{thm:Structure-of-exponential-families}\emph{} Let $(X,\,\mathcal{P})$
be an exponential family. Then there exists a Bregman divergence $D_{\psi}$
such that $(X,\,\mathcal{P_{\eta}},\,D_{\psi})$ is a dually flat
statistical manifold with respect to the Riemannian metric, induced
by $D_{\psi}$.
\end{thm}

\begin{proof}
Let $\eta$ be the natural parametrisation and $\psi$ the cumulant
generating function of $(X,\,\mathcal{P})$. Let further be $D_{\psi}$
the Bregman divergence over $(X,\,\mathcal{P})$ with respect to $\psi$,
then $(X,\,\mathcal{P_{\eta}},\,D_{\psi})$ is a Riemannian statistical
manifold with a Bregman divergence $D_{\psi}$. By lemma \ref{lem:4.3}
it follows that $\eta$ is an affine parametrisation and by lemma
\ref{lem:4.5} that the $\eta$-affine geodesics are flat with respect
to the Riemannian metric, induced by $D_{\psi}$. Let $\mu$ be the
expectation parametrisation of $(X,\,\mathcal{P})$ with respect to
$\eta$, then by lemma \ref{lem:4.1} it follows, that $\mu=\eta^{*}$,
by lemma \ref{lem:4.8}, that $\mu$ is an affine parametrisation
of $(X,\,\mathcal{P},\,D_{\psi^{*}})$ and by lemma \ref{lem:4.10}
that the $\mu$-affine geodesics are flat with respect to the Riemannian
metric, induced by $D_{\varphi}$, where $\varphi=\psi^{*}$. Therefore
the conditions of lemma \ref{lem:3.10} are satisfied and $(X,\,\mathcal{P_{\xi}},\,D_{\psi})$
is a dually flat statistical manifold.
\end{proof}
Together the primary and dual affine structure of an exponential family
induce a dually flat structure, which is characterized by the $\eta$-affine
and $\mu$-affine geodesics, with respect to the Fisher information
metric and its dual metric. Thereby the $\eta$-affine geodesics are
represented by exponential families over the curve parameter $t$
an the $\mu$-affine geodesics by mixture families. This relationship
allows a characterisation of the intrinsic dually flat structure,
which is independent of its parametrisation. This structure is given
by an \emph{$e$-affine structure}, that preserves the exponential
family representation within its primary affine structure and an \emph{$m$-affine
structure} that preserves the dual representation within the dual
affine structure. Therefore geodesics and geodesic projections in
the $e$-affine structure are respectively termed \emph{$e$-geodesics},
denoted by $\gamma_{e}$ and \emph{$e$-affine projections}, denoted
by $\pi_{e}$. Furthermore the geodesics and geodesic projections
in the $m$-affine structure are respectively termed \emph{$m$-geodesics},
denoted by $\gamma_{m}$ and \emph{$m$-affine projections}, denoted
by $\pi_{m}$. With respect to submanifolds, a smooth submanifold
is termed \emph{$e$-flat} if it has a linear embedding within the
$e$-affine structure and \emph{$m$-flat} if it has a linear embedding
within the $m$-affine structure. 
\begin{cor}
\label{cor:4.1}Let $(X,\,\mathcal{P})$ be an exponential family
and $P\in\mathcal{P}$. Then the geodesic projection of $P$ to an
$m$-flat submanifold is uniquely given by an $e$-affine projection
$\pi_{e}$ and the geodesic projection of $P$ to an $e$-flat submanifold
is uniquely given by an $m$-affine projection $\pi_{m}$. 
\begin{figure}[h]
\begin{centering}
\def\svgwidth{\columnwidth} 
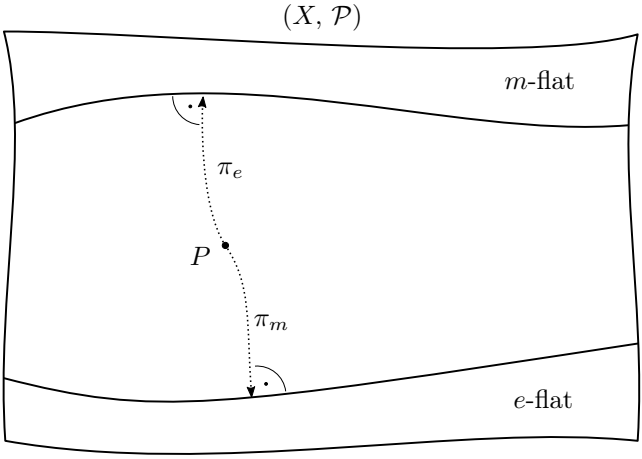
\par\end{centering}
\caption{\label{fig:cor:3.3}Unique projection in Exponential Families}
\end{figure}
\end{cor}

\begin{proof}
Let $(X,\,\mathcal{P})$ be an exponential family. Then theorem \ref{thm:Structure-of-exponential-families}
states, that the structure of $(X,\,\mathcal{P})$ is given by a dually
flat statistical manifold $(X,\,\mathcal{P}_{\eta},\,D_{\psi})$,
where the Riemannian metric is induced by a Bregman divergence $D_{\psi}$
with respect to the natural parametrisation $\eta$. Let $(X,\,\mathcal{Q})$
be an $m$-flat submanifold, then $(X,\,\mathcal{Q})$ is flat with
regard to the Riemannian metric, induced by the expectation parametrisation
and therefore flat with respect to $D_{\psi^{*}}$. Conversely let
$(X,\,\mathcal{S})$ be an $e$-flat submanifold, then $(X,\,\mathcal{Q})$
is flat with regard to the Riemannian metric, induced by the natural
parametrisation and therefore flat with respect to $D_{\psi}$. Therefore
the conditions of corollary \ref{cor:3.2} are satisfied, which proves
the corollary.
\end{proof}

\section{\label{subsec:Maximum-Entropy-Estimation}Maximum Entropy Estimation
in Exponential Families}

The sample space $(X,\,\Sigma)$ of a statistical model $(X,\,\mathcal{P})$
is generated by a statistic $T$, that induces a probability distribution
$P_{X}$ from the probability space of an underlying statistical population
$(\Omega,\,\mathcal{F},\,P)$. To this end also the sample space may
be regarded as a probability space, whereat the occurrence of the
probability distribution $P_{X}$ is hypothetical.
\begin{lem}
\label{lem:4.11}Let $(X,\,\mathcal{P}_{\eta})$ be an exponential
family over a sample space $(X,\,\Sigma)$, and $P_{U}$ the uniform
distribution over $(X,\,\Sigma)$. Then $P_{U}\in\mathcal{P}$.
\end{lem}

\begin{proof}
Let $\mu$ be a $\sigma$-finite reference measure over $(X,\,\Sigma)$,
then the probability density $u$ of the uniform distribution is given
by:
\[
u_{X}(x)\coloneqq\begin{cases}
\frac{1}{\mu(X)} & x\in X\\
0 & x\notin X
\end{cases}
\]
Let $\eta$ be the canonical parametrisation of $(X,\,\mathcal{P})$,
then the densities of any $P\in\mathcal{P}$ may be written as:
\[
p_{\eta}(\boldsymbol{x})=\exp\left(\boldsymbol{\eta}_{P}\cdot\boldsymbol{x}-\psi(\boldsymbol{\eta}_{P})\right)
\]
Where the cumulative generating function $\psi(\boldsymbol{\eta}_{P})$
is given by: 
\[
\psi(\boldsymbol{\eta}_{P})=\log\int_{X}\exp(\boldsymbol{\eta}_{P}\cdot\boldsymbol{x})\mathrm{d}\mu(\boldsymbol{x})
\]
Let $\boldsymbol{\eta}_{P}=0$, then:
\begin{eqnarray*}
p_{\eta}(\boldsymbol{x}) & = & \exp\left(0\cdot\boldsymbol{x}-\psi(0)\right)\\
 & = & \exp\left(\log\int_{X}\mathrm{d}\mu(\boldsymbol{x})\right)^{-1}\\
 & = & \mu(X)^{-1}
\end{eqnarray*}
This is the density of the uniform distribution over $(X,\,\Sigma)$.
\end{proof}
\begin{lem}
\label{lem:4.12}Let $(X,\,\Sigma)$ be a measurable space and $P_{U}$
the uniform distribution over $(X,\,\Sigma)$. Then let $P$ be an
arbitrary probability distribution over $(X,\,\Sigma)$, then:
\[
\mathrm{H}[P_{U}\mid X]\geq\mathcal{\mathrm{H}}[P\mid X]
\]
\end{lem}

\begin{proof}
Let $u_{X}$ be the probability density of $P_{U}$, then the entropy
of $P_{U}$ over $(X,\,\Sigma)$ is:
\begin{eqnarray*}
\mathrm{H}[P_{U}\mid X] & = & -\int_{X}u_{X}(x)\log u_{X}(x)\mathrm{d}\mu(x)\\
 & = & -\int_{X}\mu(X)^{-1}\log\mu(X)^{-1}\mathrm{d}\mu(x)\\
 & = & \log\mu(X)^{-1}
\end{eqnarray*}
Let $P$ be an arbitrary probability distribution over $(X,\,\Sigma)$
with density $p$, then:
\begin{eqnarray*}
D_{\mathrm{KL}}[P\parallel P_{U}] & = & \int_{X}p(x)\log\frac{p(x)}{u_{X}(x)}\mathrm{d}\mu(x)\\
 & = & \int_{X}p(x)\log p(x)\mathrm{d}\mu(x)\\
 &  & -\int_{X}p(x)\log u_{X}(x)\mathrm{d}\mu(x)\\
 & = & -\mathrm{H}[P\mid X]-\log\mu(X)^{-1}\\
 & = & -\mathcal{\mathrm{H}}[P\mid X]+\mathcal{H}[P_{U}\mid X]\geq0
\end{eqnarray*}
Therefore:
\[
\mathcal{H}[P_{U}\mid X]\geq\mathcal{H}[P\mid X],\,\forall P
\]
\end{proof}
\begin{lem}
\label{lem:4.13}Let $(X,\,\mathcal{P})$ be an exponential family
over a sample space $(X,\,\Sigma)$. Then a maximum entropy estimation
of $(X,\,\mathcal{P})$ is given by the uniform probability distribution
over $(X,\,\mathcal{P})$.
\end{lem}

\begin{proof}
From lemma \ref{lem:4.11} it follows, that $P_{U}\in\mathcal{P}$.
Furthermore lemma \ref{lem:4.12} proves that $P_{U}$ maximizes the
entropy among all $P\in\mathcal{P}$.
\end{proof}
\begin{thm}
\label{thm:MEE}Let $(X,\,\mathcal{P})$ be an exponential family
over a sample space $(X,\,\Sigma)$ and $(X,\,\mathcal{Q})$ a smooth
submanifold, then a maximum entropy estimation of $(X,\,\mathcal{Q})$
is given by a geodesic projection of the uniform probability distribution
over $(X,\,\mathcal{P})$ to $(X,\,\mathcal{Q})$.
\end{thm}

\begin{center}
\begin{figure}[h]
\begin{centering}
\def\svgwidth{\columnwidth} 
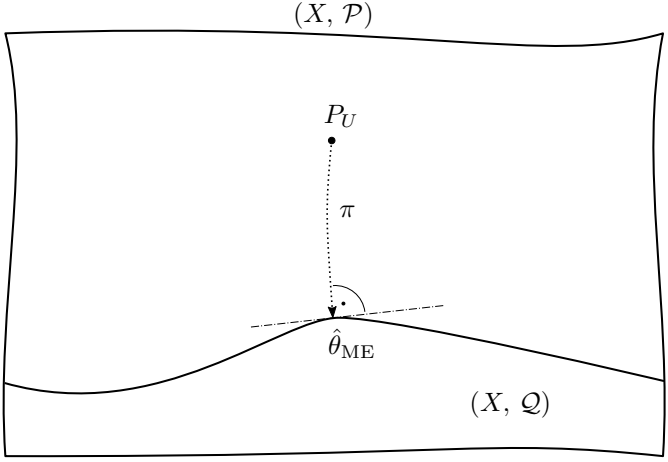
\par\end{centering}
\caption{ME estimation in Exponential Families}
\end{figure}
\par\end{center}
\begin{proof}
{[}todo{]}
\end{proof}
By assuming a given event $\sigma\in\Sigma$ within a sample space
$(X,\,\Sigma)$, the principle of maximum entropy emphasizes a probability
distribution, that minimizes additional assumptions. Then lemma \ref{lem:4.13}
states, that the maximum entropy estimation of a single observation
within the set of all probability distributions is given by the uniform
probability distribution over this observation, such that:
\[
u_{\sigma}(x)\coloneqq\begin{cases}
\frac{1}{\mu(\sigma)} & ,\,x\in\sigma\\
0 & ,\,x\notin\sigma
\end{cases}
\]
The dispensation of any additional knowledge except the observation
itself determines the uniform probability distribution as the empirical
probability of a single observation. This shall be extended to repeated
observations. Let $\boldsymbol{\sigma}=(\sigma_{i})_{i\in I}$ be
a repeated observation in $(X,\,\Sigma)$, then the density function
of a finite measure over $(X,\,\Sigma)$ is given by the arithmetic
mean of the uniform distributions of the individual single observations:
\[
f_{\boldsymbol{\sigma}}(x)=\frac{1}{|I|}\sum_{i\in I}u_{\sigma_{i}}(x)=\frac{1}{|I|}\sum_{i\in I}\frac{\delta_{x}(\sigma_{i})}{\mu(\sigma_{i})}
\]
Where $\mu$ is a $\sigma$-finite reference measure and $\delta$
the Dirac measure and defined by:
\[
\delta_{x}(\sigma)=\begin{cases}
1 & ,\,x\in\sigma\\
0 & ,\,x\notin\sigma
\end{cases}
\]
Then $f_{\boldsymbol{\sigma}}$ is a probability density function
over $(X,\,\Sigma)$, since $f_{\boldsymbol{\sigma}}(x)>0,\,\forall x\in X$
and: 
\begin{eqnarray*}
\int_{X}f_{\boldsymbol{\sigma}}(x)\mathrm{d}\mu(x) & = & \int_{X}\frac{1}{|I|}\sum_{i\in I}\frac{\delta_{x}(\sigma_{i})}{\mu(\sigma_{i})}\mathrm{d}\mu(x)\\
 & = & \frac{1}{|I|}\sum_{i\in I}\frac{1}{\mu(\sigma_{i})}\left(\int_{X}\delta_{x}(\sigma_{i})\mathrm{d}\mu(x)\right)\\
 & = & \frac{1}{|I|}\sum_{i\in I}\frac{\mu(X\cap\sigma_{i})}{\mu(\sigma_{i})}=1
\end{eqnarray*}
Furthermore the assumptions given by $f_{\boldsymbol{\sigma}}$ are
equal the knowledge which is given by the repeated observation $\boldsymbol{\sigma}$.
This determines $f_{\boldsymbol{\sigma}}$ as the density of an empirical
probability distribution.
\begin{defn*}[Empirical probability distribution]
\label{def:Empirical-probability-distribution} \emph{Let $(X,\,\Sigma)$
be a measurable space, $\boldsymbol{\sigma}=(\sigma_{i})_{i\in I}$
a repeated observation in $(X,\,\Sigma)$ and $\mu$ a $\sigma$-finite
reference measure over $(X,\,\Sigma)$. Then the empirical probability
distribution of $\boldsymbol{\sigma}$ over $(X,\,\Sigma)$ is given
by:
\[
P_{\boldsymbol{\sigma}}[A]\coloneqq\int_{A}\frac{1}{|I|}\sum_{i\in I}\frac{\delta_{x}(\sigma_{i})}{\mu(\sigma_{i})}\mathrm{d}\mu(x),\,\forall A\in\Sigma
\]
}
\end{defn*}
\begin{prop}
\label{prop:4.1}Let $(X,\,\Sigma)$ be a sample space, $P_{X}$ be
observable distribution of $(X,\,\Sigma)$ and $\boldsymbol{\sigma}(n)=(\sigma_{i})_{i\leq n}$
the partial sequences of a repeated observation $\boldsymbol{\sigma}=(\sigma_{i})_{i\in\mathbb{N}}$
in $(X,\,\Sigma)$. Then the sequence of the empirical probability
distributions $P_{\boldsymbol{\sigma}(n)}$ converges to $P_{X}$
as $n\to\infty$.
\end{prop}

\begin{proof}
Let $p_{X}$ be the probability density function of $P_{X}$. Then
due to the strong law of large numbers it follows, that:
\begin{equation}
\lim_{n\to\infty}\frac{1}{n}\sum_{i\in1}^{n}\frac{\delta_{x}(\sigma_{i})}{\mu(\sigma_{i})}\stackrel{a.s.}{=}\mathrm{E}[\delta_{x}(y)]=\int_{X}\delta_{x}(y)\mathrm{d}p_{X}(y)
\end{equation}
Let $A\in\Sigma$, then the limit of the empirical probability distributions
$P_{n}[A]$ is given by:
\begin{eqnarray*}
 &  & \lim_{n\to\infty}P_{n}[A]\\
 & = & \lim_{n\to\infty}\int_{A}\frac{1}{n}\sum_{i\in1}^{n}\frac{\delta_{\{x\}}(\sigma_{i})}{\mu(\sigma_{i})}\mathrm{d}\mu(x)\\
 & = & \int_{A}\left(\lim_{n\to\infty}\frac{1}{n}\sum_{i\in1}^{n}\frac{\delta_{\{x\}}(\sigma_{i})}{\mu(\sigma_{i})}\right)\mathrm{d}\mu(x)\\
 & \stackrel{a.s.}{=} & \int_{A}\int_{X}\delta_{\{x\}}(y)\mathrm{d}p_{X}(y)\mathrm{d}\mu(x)\\
 & = & \int_{A}p_{X}(x)\mathrm{d}\mu(x)\\
 & = & P_{X}[A]
\end{eqnarray*}
\end{proof}

\section{Maximum Likelihood Estimation in Exponential Families}
\begin{lem}
\label{lem:4.14}Let $(X,\,\mathcal{Q})$ be a statistical model over
a sample space $(X,\,\Sigma)$ and $\boldsymbol{\sigma}$ a finite
repeated observation in $(X,\,\Sigma)$. Then the MLE has the following
representation:
\begin{equation}
\hat{\theta}_{\mathrm{ML}}\in\arg\max_{\theta}\sum_{i=1}^{n}\log\mathrm{L}[Q_{\theta}\mid\sigma_{i}]\label{eq:lem:4.2:1}
\end{equation}
\end{lem}

\begin{proof}
Since the $\log$ transformation is strictly monotonous over $\mathrm{img}\mathrm{L}\subseteq[0,\,1]$
it directly follows, that:
\[
\arg\max_{\theta}\mathrm{L}[Q_{\theta}\mid\boldsymbol{\sigma}]=\arg\max_{\theta}\log\mathrm{L}[Q_{\theta}\mid\boldsymbol{\sigma}]
\]
Furthermore due to pairwise independence of the individual observations
are it follows, that:
\[
\log\mathrm{L}[Q_{\theta}\mid\boldsymbol{\sigma}]=\log\mathrm{\prod_{i=1}^{n}}\mathrm{L}[Q_{\theta}\mid\sigma_{i}]=\sum_{i=1}^{n}\log\mathrm{L}[Q_{\theta}\mid\sigma_{i}]
\]
And therefore:
\begin{eqnarray*}
\hat{\theta}_{\mathrm{ML}} & \in & \arg\max_{\theta}\mathrm{L}[Q_{\theta}\mid\boldsymbol{\sigma}]\\
 & = & \arg\max_{\theta}\sum_{i=1}^{n}\log\mathrm{L}[Q_{\theta}\mid\sigma_{i}]
\end{eqnarray*}
\end{proof}
\begin{lem}
\label{lem:4.15}Let $(X,\,\mathcal{Q})$ be a statistical model over
a sample space $(X,\,\Sigma)$ and $\boldsymbol{\sigma}$ a finite
repeated observation in $(X,\,\Sigma)$. Let further be $p_{\boldsymbol{\sigma}}$
the empirical probability density with respect to $\boldsymbol{\sigma}$.
Then the MLE has the following representation:
\begin{equation}
\hat{\theta}_{\mathrm{ML}}\in\arg\max_{\theta}\int_{X}p_{\boldsymbol{\sigma}}(x)\log q_{\theta}(x)\mathrm{d}\mu(x)\label{eq:lem:4.3:1}
\end{equation}
\end{lem}

\begin{proof}
By substitution of the empirical probability density $p_{\boldsymbol{\sigma}}(x)$
it follows that:
\begin{eqnarray}
 &  & \int_{X}p_{\boldsymbol{\sigma}}(x)\log q_{\theta}(x)\mathrm{d}\mu(x)\label{eq:lem:4.3:proof:1}\\
 & = & \int_{X}\left(\frac{1}{n}\sum_{i\in1}^{n}\frac{1}{\mu(\sigma_{i})}\delta_{\{x\}}(\sigma_{i})\right)\log q_{\theta}(x)\mathrm{d}\mu(x)\nonumber \\
 & = & \frac{1}{n}\sum_{i\in1}^{n}\frac{1}{\mu(\sigma_{i})}\int_{\sigma_{i}}\log q_{\theta}(x)\mathrm{d}\mu(x)\nonumber \\
 & = & \frac{1}{n}\sum_{i\in1}^{n}\frac{1}{\mu(\sigma_{i})}\left(\mu(\sigma_{i})\log\mathrm{L}[Q_{\theta}\mid\sigma_{i}]-\mu(\sigma_{i})\right)\nonumber \\
 & = & \frac{1}{n}\sum_{i\in1}^{n}\log\mathrm{L}[Q_{\theta}\mid\sigma_{i}]-1\nonumber 
\end{eqnarray}
The maximization of equation \ref{eq:lem:4.3:proof:1} with respect
to $\theta$ therefore yields the following identity:
\begin{eqnarray*}
 &  & \arg\max_{\theta}\int_{X}p_{\boldsymbol{\sigma}}(x)\log q_{\theta}(x)\mathrm{d}\mu(x)\\
 & = & \arg\max_{\theta}\left(\frac{1}{n}\sum_{i\in I}\log\mathcal{\mathrm{L}}[Q_{\theta}\mid\sigma_{i}]-1\right)\\
 & = & \arg\max_{\theta}\sum_{i\in I}\log\mathrm{L}[Q_{\theta}\mid\sigma_{i}]
\end{eqnarray*}
By lemma \ref{lem:4.14}, equation \ref{eq:lem:4.2:1} it follows,
that:
\[
\hat{\theta}_{\mathrm{ML}}\in\arg\max_{\theta}\int_{X}p_{\boldsymbol{\sigma}}(x)\log q_{\theta}(x)\mathrm{d}\mu(x)
\]
\end{proof}
\begin{thm}
\label{thm:4.3}Let $(X,\,\mathcal{P})$ be an exponential family
over a sample space $(X,\,\Sigma)$. Let further be $(X,\,\mathcal{Q})$
a smooth submanifold of $(X,\,\mathcal{P})$ and $\boldsymbol{\sigma}$
a repeated observation in $(X,\,\Sigma)$. Then a maximum likelihood
estimation of $(X,\,\mathcal{Q})$ respective to $\boldsymbol{\sigma}$
is given by the geodesic projection of the empirical probability $P_{\boldsymbol{\sigma}}$
to $(X,\,\mathcal{Q})$ in $(X,\,\mathcal{P})$.
\end{thm}

\begin{center}
\begin{figure}[h]
\begin{centering}
\def\svgwidth{\columnwidth} 
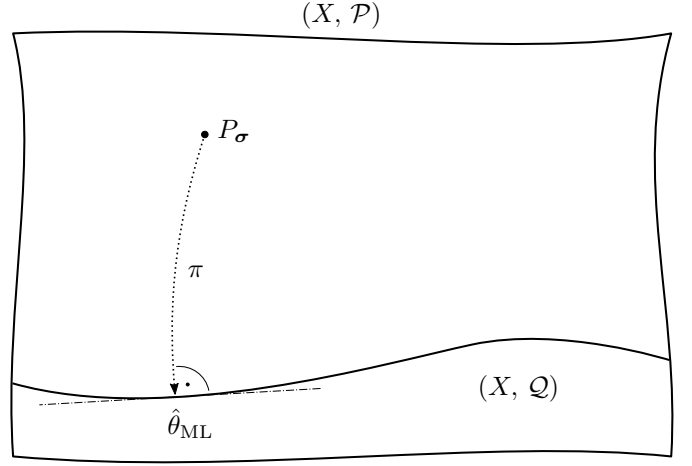
\par\end{centering}
\caption{ML estimation in Exponential Families}
\end{figure}
\par\end{center}
\begin{proof}
Let $(X,\,\mathcal{P}_{\eta})$ be an exponential family over a sample
space $(X,\,\Sigma)$ and $(X,\,\mathcal{Q})$ a smooth submanifold
of $(X,\,\mathcal{P})$. Then there exists a Bregman divergence $D_{\psi}$,
such that $(X,\,\mathcal{P}_{\eta},\,D_{\psi})$ is a dually flat
statistical manifold and $\eta$ is the $e$-affine parametrisation
of $(X,\,\mathcal{P}_{\eta},\,D_{\psi})$. Then $(X,\,\mathcal{Q}_{\eta},\,D_{\psi})$
is a smooth Riemannian submanifold of $(X,\,\mathcal{P}_{\eta},\,D_{\psi})$
with respect to the induced Riemannian metric $D_{\psi}$. Let $\boldsymbol{\sigma}$
be a repeated observation over $(X,\,\Sigma)$ and $P_{\boldsymbol{\sigma}}$
its empirical probability distribution.

Then $P_{\boldsymbol{\sigma}}$ is given in expectation parameters
in $(X,\,\Sigma)$ and therefore in an $m$-affine parametrisation
within Riemannian Manifold $(X,\,\mathcal{P}_{\eta},\,D_{\psi})$.
By applying the projection theorem the geodesic projection of $P_{\boldsymbol{\sigma}}$
to $(X,\,\mathcal{Q}_{\eta},\,D_{\psi})$, equals the dual affine
projection, and therefore by a point $Q\in\mathcal{Q}$, that minimizes
the Bregman divergence $D_{\psi}[Q\parallel P_{\boldsymbol{\sigma}}]$.
The geodesic distance is therefore given by:
\begin{eqnarray*}
 &  & d(P_{\boldsymbol{\sigma}},\,Q)\\
 & = & D_{\psi}[Q\parallel P_{\boldsymbol{\sigma}}]\\
 & = & D_{\mathrm{KL}}[P_{\boldsymbol{\sigma}}\parallel Q]\\
 & = & \int_{X}p_{\boldsymbol{\sigma}}(x)\log\frac{p_{\boldsymbol{\sigma}}(x)}{q(x)}\mathrm{d}\mu(x)\\
 & = & \int_{X}p_{\boldsymbol{\sigma}}(x)\log p_{\boldsymbol{\sigma}}(x)\mathrm{d}\mu(x)\\
 &  & -\int_{X}p_{\boldsymbol{\sigma}}(x)\log q(x)\mathrm{d}\mu(x)\\
 & = & -\mathrm{H}[P_{\sigma}\mid X]-\int_{X}p_{\boldsymbol{\sigma}}(x)\log q(x)\mathrm{d}\mu(x)
\end{eqnarray*}
The minimization of $d(P_{\boldsymbol{\sigma}},\,Q)$ with respect
to the natural parametrisation $\eta$ therefore yields the following
identity:
\begin{equation}
\arg\min_{\eta}d(P_{\boldsymbol{\sigma}},\,Q_{\eta})=\arg\max_{\eta}\int_{X}p_{\boldsymbol{\sigma}}(x)\log q_{\eta}(x)\mathrm{d}\mu(x)\label{eq:prop:4.1:1}
\end{equation}
By equation $\ref{eq:lem:4.3:1}$ and lemma, equation $\ref{eq:prop:4.1:1}$
it follows, that: 
\begin{eqnarray*}
\hat{\theta}_{\mathrm{ML}} & \stackrel{\ref{eq:lem:4.3:1}}{\in} & \arg\max_{\eta}\int_{X}p_{\boldsymbol{\sigma}}(x)\log q_{\eta}(x)\mathrm{d}\mu(x)\\
 & \stackrel{\ref{eq:prop:4.1:1}}{=} & \arg\min_{\eta}d(P_{\boldsymbol{\sigma}},\,Q_{\eta})
\end{eqnarray*}
\end{proof}

\section{Latent variable models}

Exponential families, as introduced in the previous sections, statistically
relate random variables over a common statistical population by their
common probability distribution in the sample space. In many cases
however the intrinsic structure of this relationship has a natural
decomposition by the introduction of latent random variables, that
are not directly observable from the statistical population but assumed
to affect the observations. This is of particular importance for the
modelling of statistical populations with complex network structures.
In this case the properties of the network may be incorporated by
the conditional transition probabilities between observable and latent
random variables.

In completely observable statistical models, the probability distributions
may be estimated by the empirical probability distributions of repeated
observations. In latent variable models however the conditional transition
probabilities $p(v\mid h)$ and $p(h\mid v)$ between the observables
$v\in V$ and the latent variables $h\in H$ in general prevent this
inference. The only exception is given if $p(v\mid h)$ and $p(h\mid v)$
are uniform distributed, such that for any given $v$ any $h$ has
the same probability with respect to $v$ and vice versa. In this
case estimations decompose into independent estimations of the observable
variables and the latent variables. If the conditional transition
probabilities, however are not uniform distributed, they have to be
taken into account for estimations. This also applies to empirical
distributions. Let $\boldsymbol{\sigma}(n)=(\sigma_{i})_{i\leq n}$
be the partial sequences of a repeated observation $\boldsymbol{\sigma}=(\sigma_{i})_{i\in\mathbb{N}}$
in $(V,\,\Sigma_{v})$, then by proposition \ref{prop:4.1} it follows,
that the empirical probabilities $P_{\boldsymbol{\sigma}(n)}$ converge
in distribution to the true probability distribution $P_{V}$ of $(V,\,\Sigma_{v})$.
Since $P_{V}$ however is the marginal distribution of the observables
in \emph{$(X,\,\Sigma)$} the common empirical probabilities over
\emph{$(X,\,\Sigma)$} are constituted by an empirical probability
of a repeated observation and a conditional transition probability.
The probability density $p_{\boldsymbol{\sigma}}$ of an empirical
probability over $(X,\,\Sigma)$ is therefore given by:
\[
p_{\boldsymbol{\sigma}}(x)=p_{\boldsymbol{\sigma}}(v,\,h)=p_{\boldsymbol{\sigma}}(v)p(h\mid v),\,\forall x\in X
\]
In the presence of continuous variables it is useful to restrict the
empirical probability distributions to ``non pathological'' cases.
This restriction defines statistical model with respect to empirical
observations.
\begin{defn*}[Empirical model]
\label{def:Empirical-model}\emph{ Let $(X,\,\Sigma)$ be a partially
observable measurable space with $X=V\times H$. Then a statistical
model $(X,\,\mathcal{E})$ is termed an empirical model over $(X,\,\Sigma)$
if $\mathcal{E}$ comprises all empirical probability distributions
over $(X,\,\Sigma)$, which are constituted by a finite repeated observation
and a conditional transition probability of a given set $\mathcal{T}$.
If $\mathcal{T}$ is the set of all absolutely continuous conditional
transition probabilities, then $(X,\,\mathcal{E})$ is termed an absolutely
continuous empirical model.}
\end{defn*}

\section{Exponential families with latent variables}

Let $(X,\,\mathcal{P})$ be an exponential family over a partially
observable measurable space $(X,\,\mathcal{P})$. Then due to theorem
\ref{thm:Structure-of-exponential-families} the structure of $(X,\,\mathcal{P})$
is that of a dually flat manifold such that there exists a parametrisation
$\eta$ and a convex function $\psi$, that allow to regard $(X,\,\mathcal{P})$
as a dually flat statistical manifold, given by $(X,\,\mathcal{P}_{\eta},\,D_{\psi})$.
In the purpose of observation based estimations in the presence of
latent variables, the obstacle that has to be taken is the extension
of $(X,\,\mathcal{P}_{\eta},\,D_{\psi})$ to a dually flat embedding
space $(X,\,\mathcal{U},\,D)$, that covers $(X,\,\mathcal{P}_{\eta},\,D_{\psi})$
as well as the empirical model $(X,\,\mathcal{E})$. For arbitrary
empirical models however this embedding does generally not exist,
for which $(X,\,\mathcal{E})$ is assumed to be absolutely continuous.
Then $(X,\,\mathcal{E})$ is generally an infinite dimensional exponential
family.
\begin{lem}
\label{lem:4.16}Let $(X,\,\mathcal{E})$ be an absolutely continuous
empirical model over a partially observable measurable space $(X,\,\Sigma)$.
Then $(X,\,\mathcal{E})$ is an infinite dimensional exponential family
and its probability densities are given by:
\begin{align*}
p_{\boldsymbol{\sigma}}(\boldsymbol{v},\,\boldsymbol{h}) & =\frac{\exp\int_{H}\eta_{s}(\boldsymbol{v},\,\boldsymbol{r})\delta(\boldsymbol{h}-\boldsymbol{r})\mathrm{d}\boldsymbol{r}}{\exp\psi_{s}(\eta_{s}(\boldsymbol{v},\,\boldsymbol{h}))}
\end{align*}
where $\boldsymbol{v}\mathbb{\in R}^{k}$ and $\boldsymbol{h}\in\mathbb{R}^{l}$
respectively denote the observable and latent variables, $\eta_{s}(\boldsymbol{v},\,\boldsymbol{h})$
scalar coefficients and $\psi_{s}$ the cumulant generating function,
given by:
\begin{equation}
\psi_{s}(\eta_{s}(\boldsymbol{v},\,\boldsymbol{h}))=\log\int_{\mathbb{R}^{l}}\exp(\eta_{s}(\boldsymbol{v},\,\boldsymbol{r}))\mathrm{d}\boldsymbol{r}\label{eq:lvm:Lemma41}
\end{equation}
\end{lem}

\begin{proof}
Due to the definition of an empirical model the probability density
$p_{\boldsymbol{\sigma}}$ of any $P\in\mathcal{E}$ may be written
as:
\begin{equation}
p_{\boldsymbol{\sigma}}(\boldsymbol{v},\,\boldsymbol{h})=\frac{1}{n}\sum_{i=1}^{n}\delta(\boldsymbol{v}-\boldsymbol{s}_{i})p(\boldsymbol{h}\mid\boldsymbol{s}_{i})\label{eq:lvm:Lemma41:eq1}
\end{equation}
Furthermore any conditional transition probability $p(\boldsymbol{h}\mid\boldsymbol{v})$
is given by:
\begin{equation}
p(\boldsymbol{h}\mid\boldsymbol{v})=\int_{H}\delta(\boldsymbol{h}-\boldsymbol{r})p(\boldsymbol{r}\mid\boldsymbol{v})\mathrm{d}\boldsymbol{r}\label{eq:lvm:Lemma41:eq2}
\end{equation}
Therefore by equation \ref{eq:lvm:Lemma41:eq1} and \ref{eq:lvm:Lemma41:eq2}
it follows, that:
\begin{equation}
p_{s}(\boldsymbol{v},\,\boldsymbol{h})=\frac{1}{n}\sum_{i=1}^{n}\delta(\boldsymbol{v}-\boldsymbol{s}_{i})\int_{\mathbb{R}^{l}}\delta(\boldsymbol{h}-\boldsymbol{r})p(\boldsymbol{r}\mid\boldsymbol{s}_{i})\mathrm{d}\boldsymbol{r}\label{eq:lvm:Lemma41:eq3}
\end{equation}
Since $p(\boldsymbol{h}\mid\boldsymbol{v})$ are absolutely continues
equation \ref{eq:lvm:Lemma41:eq3} may be rewritten to: 
\[
p_{s}(\boldsymbol{v},\,\boldsymbol{h})=\int_{H}\left(\frac{1}{n}\sum_{i=1}^{n}\delta(\boldsymbol{v}-\boldsymbol{s}_{i})p(\boldsymbol{r}\mid\boldsymbol{s}_{i})\right)\delta(\boldsymbol{h}-\boldsymbol{r})\mathrm{d}\boldsymbol{r}
\]
By the substitution
\begin{equation}
\mu_{s}(\boldsymbol{v},\,\boldsymbol{r})=\frac{1}{n}\sum_{i=1}^{n}\delta(\boldsymbol{v}-\boldsymbol{s}_{i})p(\boldsymbol{r}\mid\boldsymbol{s}_{i})\label{eq:lvm:Lemma41:eq4}
\end{equation}
it follows, that the the empirical probabilities in $\mathcal{E}$
may be written as a mixture with mixing coefficients $\mu_{s}(\boldsymbol{v},\,\boldsymbol{r})$:
\begin{equation}
p_{s}(\boldsymbol{v},\,\boldsymbol{h})=\int_{H}\mu_{s}(\boldsymbol{v},\,\boldsymbol{r})\delta(\boldsymbol{h}-\boldsymbol{r})\mathrm{d}\boldsymbol{r}\label{eq:lvm:Lemma41:eq5}
\end{equation}
This shows, that $\mathcal{E}$ is an infinite dimensional mixture
family. By a further transformation:
\[
\eta_{s}(\boldsymbol{v},\,\boldsymbol{h})=\log\mu_{s}(\boldsymbol{v},\,\boldsymbol{h})+\psi_{s}
\]
 a dual representation of equation \ref{eq:lvm:Lemma41:eq5} is obtained,
with:
\begin{align}
p_{s}(\boldsymbol{v},\,\boldsymbol{h})= & \frac{\exp\int_{\mathbb{R}^{l}}\eta_{s}(\boldsymbol{v},\,\boldsymbol{r})\delta(\boldsymbol{h}-\boldsymbol{r})\mathrm{d}\boldsymbol{r}}{\exp\psi_{s}(\eta_{s}(\boldsymbol{v},\,\boldsymbol{h}))}\label{eq:lvm:Lemma41:eq6}
\end{align}
This shows, that $(X,\,\mathcal{E})$ is also an infinite dimensional
exponential family with natural variables $\eta_{s}(\boldsymbol{v},\,\boldsymbol{r})$
and the cumulative generating function is given by
\[
\text{ }\psi_{s}(\eta_{s}(\boldsymbol{v},\,\boldsymbol{h}))=\log\int_{\mathbb{R}^{l}}\exp(\eta_{s}(\boldsymbol{v},\,\boldsymbol{r}))\mathrm{d}\boldsymbol{r}
\]
\end{proof}
\begin{lem}
\label{lem:4.17}Let $(X,\,\mathcal{P})$ be an exponential family
over a partially observable measurable space $(X,\,\Sigma)$ and $(X,\,\mathcal{E})$
an absolutely continuous empirical model over $(X,\,\Sigma)$. Then
there exists a dually flat statistical manifold $(X,\,\mathcal{U},\,D)$,
that covers $(X,\,\mathcal{P})$ and $(X,\,\mathcal{E})$ as dually
flat submanifolds.
\end{lem}

\begin{proof}
With respect to the partially observable measurable space $(X,\,\Sigma)$
any $p\in\mathcal{P}$ may be written as $p(\boldsymbol{x})=p(\boldsymbol{v},\,\boldsymbol{h})$
and any $p_{\boldsymbol{s}}\in\mathcal{E}$ as $p_{s}(\boldsymbol{x})=p_{s}(\boldsymbol{v},\,\boldsymbol{h})$.
Let $\mathcal{U}$ be given by:
\[
\mathcal{U}=\left\{ q\mid q(\boldsymbol{v},\,\boldsymbol{h})=p(\boldsymbol{v},\,\boldsymbol{h})p_{s}(\boldsymbol{v},\,\boldsymbol{h}),\,p\in\mathcal{P},\,p_{s}\in\mathcal{E}\right\} 
\]
Since $(X,\,\mathcal{P})$ is an exponential family it has a natural
parametrisation $\boldsymbol{\eta}_{\mathcal{P}}$ and due to Lemma
\ref{lem:4.16} the continuous empirical model $(X,\,\mathcal{E})$
is an infinite dimensional exponential family with natural coefficients
$\boldsymbol{\eta}_{\mathcal{E}}=(\eta_{s}(\boldsymbol{v},\,\boldsymbol{h})_{v,h})^{T}$.
Then a parametrisation of $(X,\,\mathcal{U})$ is given by:
\[
\boldsymbol{\eta}=(\boldsymbol{\eta}_{\mathcal{P}},\,\boldsymbol{\eta}_{\mathcal{E}})^{T}
\]
Let further be $\psi_{\mathcal{P}}$ the cumulant generating function
of $(X,\,\mathcal{P})$ and $\psi_{\mathcal{E}}$ the cumulant generating
function of $(X,\,\mathcal{E})$, then $\psi_{\mathcal{P}}$ and $\psi_{\mathcal{E}}$
are convex functions and therefore a convex function over $\boldsymbol{\eta}$
is given by:
\[
\psi(\boldsymbol{\eta})=\psi_{\mathcal{P}}(\boldsymbol{\eta}_{\mathcal{P}})+\psi_{\mathcal{E}}(\boldsymbol{\eta}_{\mathcal{E}})
\]
This allows the definition of a Bregman divergence $D_{\psi}$, such
that $D_{\psi}$ induces Riemannian metric over $(X,\,\mathcal{U})$.
By substitution of equation \ref{eq:lvm:Lemma41:eq6} it follows that
the application of $\boldsymbol{\eta}$ to $(X,\,\mathcal{U})$ yields
a parametric representation, which is given by:
\begin{align*}
q_{\boldsymbol{\eta}}(\boldsymbol{v},\,\boldsymbol{h}) & =\exp\left(\int_{\mathbb{R}^{l}}\eta_{s}(\boldsymbol{v},\,\boldsymbol{r})\delta(\boldsymbol{h}-\boldsymbol{r})\mathrm{d}\boldsymbol{r}\right.\\
 & \left.+\boldsymbol{\eta}_{\mathcal{P}}\cdot(\boldsymbol{v},\,\boldsymbol{h})^{T}-\psi(\boldsymbol{\eta})\right)
\end{align*}
This shows that $(X,\,\mathcal{U}_{\boldsymbol{\eta}})$ is an exponential
family and furthermore, that $(X,\,\mathcal{U}_{\boldsymbol{\eta}},\,D_{\psi})$
is a dually flat statistical manifold, that covers $(X,\,\mathcal{P})$
and $(X,\,\mathcal{E})$ as smooth submanifolds. Since the projections
$\boldsymbol{\eta}\to\boldsymbol{\eta}_{\mathcal{P}}$ and $\boldsymbol{\eta}\to\boldsymbol{\eta}_{\mathcal{E}}$
are linear in the $e$-parametrisation $(X,\,\mathcal{P})$ and $(X,\,\mathcal{E})$
are $e$-flat with respect to the induced metric $D_{\psi}$. Let
$\boldsymbol{\mu}_{\mathcal{P}}$ be the expectation parameters of
$(X,\,\mathcal{P})$ and $\boldsymbol{\mu}_{\mathcal{E}}=(\mu_{s}(\boldsymbol{v},\,\boldsymbol{h})_{v,h})^{T}$
expectation coefficients $(X,\,\mathcal{E})$, then the dual parametrisation
$\boldsymbol{\mu}$ is given by:
\begin{eqnarray*}
\boldsymbol{\mu} & = & \nabla\psi_{\mathcal{P}}(\boldsymbol{\eta}_{\mathcal{P}})+\nabla\psi_{\mathcal{E}}(\boldsymbol{\eta}_{\mathcal{E}})\\
 & = & (\boldsymbol{\mu}_{\mathcal{P}},\,0)^{T}+(0,\,\boldsymbol{\mu}_{\mathcal{E}})^{T}\\
 & = & (\boldsymbol{\mu}_{\mathcal{P}},\,\boldsymbol{\mu}_{\mathcal{E}})^{T}
\end{eqnarray*}
Since the projections $\boldsymbol{\mu}\to\boldsymbol{\mu}_{\mathcal{P}}$
and $\boldsymbol{\mu}\to\boldsymbol{\mu}_{\mathcal{E}}$ are linear
in the $m$-parametrisation $(X,\,\mathcal{P})$ and $(X,\,\mathcal{E})$
are $m$-flat with respect to the induced metric $D_{\psi}$. Therefore
$(X,\,\mathcal{P},\,D_{\psi})$ and $(X,\,\mathcal{E},\,D_{\psi})$
are a dually flat submanifolds of $(X,\,\mathcal{U},\,D_{\psi})$.
\end{proof}

\section{Maximum Likelihood Estimation in Exponential Families with Latent
Variables}

Due to the existence of a simply connected embedding space $(X,\,\mathcal{U},\,D_{\varphi})$
that covers $(X,\,\mathcal{P})$ as well as $(X,\,\mathcal{E})$,
also arbitrary smooth submanifolds $(X,\,\mathcal{Q})$ of $(X,\,\mathcal{P})$
may be connected to submanifolds $(X,\,\mathcal{E}_{\boldsymbol{\sigma}})$
of $(X,\,\mathcal{E})$, given by a repeated observation $\boldsymbol{\sigma}$.
Since $(X,\,\mathcal{U},\,D_{\varphi})$ is furthermore a Riemannian
statistical manifold and $(X,\,\mathcal{Q})$ and $(X,\,\mathcal{E}_{\boldsymbol{\sigma}})$
are smooth submanifolds of $(X,\,\mathcal{U},\,D_{\varphi})$ geodesics
between $(X,\,\mathcal{Q})$ and $(X,\,\mathcal{E}_{\boldsymbol{\sigma}})$
are given by the induced Riemannian metric $D_{\varphi}$. 
\begin{thm}
\label{thm:4.2}Let $(X,\,\mathcal{P})$ be an exponential family
over a partially observable measurable space $(X,\,\Sigma)$ and $(X,\,\mathcal{E})$
an absolutely continuous empirical model over $(X,\,\Sigma)$. Let
further be $(X,\,\mathcal{Q})$ a smooth submanifold of $(X,\,\mathcal{P})$,
and $\boldsymbol{\sigma}$ a repeated observation in $(X,\,\Sigma)$.
Then a maximum likelihood estimation of $(X,\,\mathcal{Q})$ respective
to $\boldsymbol{\sigma}$ is given by a minimal geodesic projection
of $(X,\,\mathcal{E}_{\boldsymbol{\sigma}})$ to $(X,\,\mathcal{Q})$.
\end{thm}

\begin{center}
\begin{figure}[h]
\begin{centering}
\def\svgwidth{\columnwidth} 
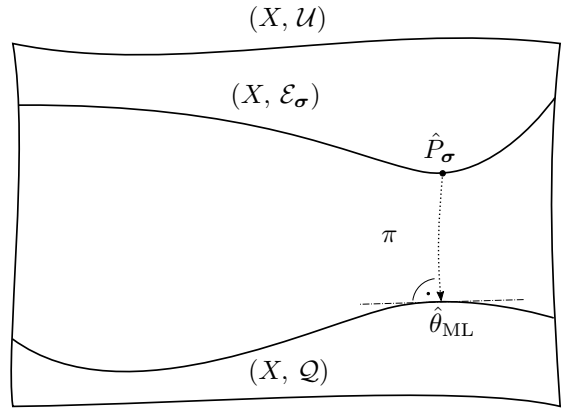
\par\end{centering}
\caption{ML estimation in latent variable Exponential Families}
\end{figure}
\par\end{center}
\begin{proof}
Since $(X,\,\mathcal{P})$ is an exponential family and $(X,\,\mathcal{E})$
a continuous empirical model, there exists a common dually flat embedding
space $(X,\,\mathcal{U}_{\eta},\,D_{\psi})$, such that $(X,\,\mathcal{P})$
and $(X,\,\mathcal{E})$ are dually flat submanifolds, with respect
to the Bregman divergence $D_{\psi}$. Since the maximum likelihood
estimation of $\mathcal{Q}$ with respect to the repeated observation
$\boldsymbol{\sigma}$ is independent of its chosen parametrisation,
it may be obtained by the natural parametrisation of the embedding
space $(X,\,\mathcal{U}_{\eta},\,D_{\psi})$, such that:\emph{
\[
\hat{\theta}_{ML}\in\arg\max_{\eta}\mathrm{L}[Q_{\eta}\mid\boldsymbol{\sigma}]
\]
}Then theorem has an equivalent formulation, given by:
\begin{equation}
\arg\max_{\eta}\mathrm{L}[Q_{\eta}\mid\boldsymbol{\sigma}]=\arg\min_{\eta}d(P_{\eta},\,Q_{\eta})\label{eq:thm:4.1:1}
\end{equation}
Since $(X,\,\mathcal{Q})$ is a smooth submanifold of $(X,\,\mathcal{P})$
and $(X,\,\mathcal{P})$ of $(X,\,\mathcal{U})$ it follows, that
$(X,\,\mathcal{Q})$ is also a smooth submanifold of $(X,\,\mathcal{U})$
and since $(X,\,\mathcal{U},\,D_{\psi})$ is a dually flat statistical
manifold, the projection theorem is satisfied. Therefore the geodesic
projection $\pi:\mathcal{E}_{\boldsymbol{\sigma}}\to\mathcal{Q}$
of a fixed point $P\in\mathcal{E}_{\boldsymbol{\sigma}}$ to $\mathcal{Q}$
is given by the minimal dual affine projection, and thus by a point
$Q_{\eta}\in\mathcal{Q}$, that minimizes $D_{\psi}[Q_{\eta}\parallel P]$,
such that: 
\[
\arg\min_{\eta}d(P,\,Q_{\eta})=\arg\min_{\eta}D_{\psi}[Q_{\eta}\parallel P]
\]
Then the minimal geodesic projection of $\mathcal{E}_{\boldsymbol{\sigma}}$
to $\mathcal{Q}$ is given by points $P_{\boldsymbol{\eta}}\in\mathcal{E}_{\boldsymbol{\sigma}}$
and $Q_{\boldsymbol{\eta}}\in\mathcal{Q}$, that minimize $D_{\psi}[Q_{\boldsymbol{\eta}}\parallel P_{\boldsymbol{\eta}}]$:
\begin{equation}
\arg\min_{\boldsymbol{\eta}}d(P_{\boldsymbol{\eta}},\,Q_{\boldsymbol{\eta}})=\arg\min_{\boldsymbol{\eta}}D_{\psi}[Q_{\boldsymbol{\eta}}\parallel P_{\boldsymbol{\eta}}]\label{eq:thm:4.1:2}
\end{equation}
Since $\boldsymbol{\eta}$ is the natural parametrisation of the exponential
family $(X,\,\mathcal{U})$, the Bregman divergence $D_{\psi}$ is
given by the dual Kullback-Leibler divergence in natural parameters.
Therefore it follows, that:
\begin{equation}
\arg\min_{\boldsymbol{\eta}}D_{\psi}[Q_{\boldsymbol{\eta}}\parallel P_{\boldsymbol{\eta}}]=\arg\min_{\boldsymbol{\eta}}D_{\mathrm{KL}}[P_{\boldsymbol{\eta}}\parallel Q_{\boldsymbol{\eta}}]\label{eq:thm:4.1:3}
\end{equation}
Without loss of generality let $\boldsymbol{v}:(X_{v},\,\Sigma_{v})\to V$
be the vectorial observable random variable and $\boldsymbol{h}:(X_{h},\,\Sigma_{h})\to H$
the vectorial latent random variable. Then $\mathrm{d}\boldsymbol{v}$
and $\mathrm{d}\boldsymbol{h}$ denote the Lebesgue measures in $V$
and $H$ and:
\begin{eqnarray*}
D_{\mathrm{KL}}[P_{\boldsymbol{\eta}}\parallel Q_{\boldsymbol{\eta}}] & = & \int_{X}p_{\boldsymbol{\eta}}(\boldsymbol{x})\log\frac{p_{\boldsymbol{\eta}}(\boldsymbol{x})}{q_{\boldsymbol{\eta}}(\boldsymbol{x})}\mathrm{d}\boldsymbol{x}\\
 & = & \int_{H,\boldsymbol{\sigma}}p_{\boldsymbol{\eta}}(\boldsymbol{v},\,\boldsymbol{h})\log\frac{p_{\boldsymbol{\eta}}(\boldsymbol{v},\,\boldsymbol{h})}{q_{\boldsymbol{\eta}}(\boldsymbol{v},\,\boldsymbol{h})}\mathrm{d}\boldsymbol{v}\mathrm{d}\boldsymbol{h}
\end{eqnarray*}
The common empirical distributions $p_{\boldsymbol{\eta}}(\boldsymbol{v},\,\boldsymbol{h})$
are defined by the marginal empirical densities $p_{\boldsymbol{\eta}}(\boldsymbol{v})$
of the observable variables and conditional transition probabilities
$p_{\boldsymbol{\eta}}(\boldsymbol{h}\mid\boldsymbol{v})$ of the
latent variables by $p_{\boldsymbol{\eta}}(\boldsymbol{v},\,\boldsymbol{h})=p_{\boldsymbol{\eta}}(\boldsymbol{v})p_{\boldsymbol{\eta}}(\boldsymbol{h}\mid\boldsymbol{v})$,
such that:
\begin{eqnarray*}
 &  & D_{\mathrm{KL}}[P_{\boldsymbol{\eta}}\parallel Q_{\boldsymbol{\eta}}]\\
 & = & \int_{H,\boldsymbol{\sigma}}p_{\boldsymbol{\eta}}(\boldsymbol{v})p_{\boldsymbol{\eta}}(\boldsymbol{h}\mid\boldsymbol{v})\log\frac{p_{\boldsymbol{\eta}}(\boldsymbol{v})p_{\boldsymbol{\eta}}(\boldsymbol{h}\mid\boldsymbol{v})}{q_{\boldsymbol{\eta}}(\boldsymbol{v},\,\boldsymbol{h})}\mathrm{d}\boldsymbol{v}\mathrm{d}\boldsymbol{h}
\end{eqnarray*}
Furthermore the $\log$ function allows to write the product into
a sum and to substitute the addends by functionals, such that:
\begin{eqnarray*}
 &  & D_{\mathrm{KL}}[P_{\boldsymbol{\eta}}\parallel Q_{\boldsymbol{\eta}}]\\
 & = & F_{1}(P_{\boldsymbol{\eta}},\,Q_{\boldsymbol{\eta}})+F_{2}(P_{\boldsymbol{\eta}},\,Q_{\boldsymbol{\eta}})+F_{3}(P_{\boldsymbol{\eta}},\,Q_{\boldsymbol{\eta}})
\end{eqnarray*}
With:
\begin{eqnarray*}
 &  & F_{1}(P_{\boldsymbol{\eta}},\,Q_{\boldsymbol{\eta}})\\
 &  & \coloneqq\int_{H,\boldsymbol{\sigma}}p_{\boldsymbol{\eta}}(\boldsymbol{v})p_{\boldsymbol{\eta}}(\boldsymbol{h}\mid\boldsymbol{v})\log p_{\boldsymbol{\eta}}(\boldsymbol{v})\mathrm{d}\boldsymbol{v}\mathrm{d}\boldsymbol{h}\\
 &  & F_{2}(P_{\boldsymbol{\eta}},\,Q_{\boldsymbol{\eta}})\\
 &  & \coloneqq\int_{H,\boldsymbol{\sigma}}p_{\boldsymbol{\eta}}(\boldsymbol{v})p_{\boldsymbol{\eta}}(\boldsymbol{h}\mid\boldsymbol{v})\log p_{\boldsymbol{\eta}}(\boldsymbol{h}\mid\boldsymbol{v})\mathrm{d}\boldsymbol{v}\mathrm{d}\boldsymbol{h}\\
 &  & F_{3}(P_{\boldsymbol{\eta}},\,Q_{\boldsymbol{\eta}})\\
 &  & \coloneqq-\int_{H,\boldsymbol{\sigma}}p_{\boldsymbol{\eta}}(\boldsymbol{v})p_{\boldsymbol{\eta}}(\boldsymbol{h}\mid\boldsymbol{v})\log q_{\boldsymbol{\eta}}(\boldsymbol{v},\,\boldsymbol{h})\mathrm{d}\boldsymbol{v}\mathrm{d}\boldsymbol{h}
\end{eqnarray*}
Then it follows, that:
\begin{eqnarray*}
 &  & F_{1}(P_{\boldsymbol{\eta}},\,Q_{\boldsymbol{\eta}})\\
 &  & =\int_{\boldsymbol{\sigma}}p_{\boldsymbol{\eta}}(\boldsymbol{v})\log p_{\boldsymbol{\eta}}(\boldsymbol{v})\left(\int_{H}p_{\boldsymbol{\eta}}(\boldsymbol{h}\mid\boldsymbol{v})\mathrm{d}\boldsymbol{h}\right)\mathrm{d}\boldsymbol{v}\\
 &  & =\int_{\boldsymbol{\sigma}}p_{\boldsymbol{\eta}}(\boldsymbol{v})\log p_{\boldsymbol{\eta}}(\boldsymbol{v})\mathrm{d}\boldsymbol{v}\\
 &  & =-\mathrm{H}[P_{\eta}\mid\boldsymbol{\sigma}]
\end{eqnarray*}
And furthermore::
\begin{eqnarray*}
 &  & F_{2}(P_{\boldsymbol{\eta}},\,Q_{\boldsymbol{\eta}})\\
 &  & =\int_{\boldsymbol{\sigma}}p_{\boldsymbol{\eta}}(\boldsymbol{v})\left(\int_{H}p_{\boldsymbol{\eta}}(\boldsymbol{h}\mid\boldsymbol{v})\log p_{\boldsymbol{\eta}}(\boldsymbol{h}\mid\boldsymbol{v})\mathrm{d}\boldsymbol{h}\right)\mathrm{d}\boldsymbol{v}\\
 &  & =-\int_{\boldsymbol{\sigma}}p_{\boldsymbol{\eta}}(\boldsymbol{v})\mathrm{H}[p(\boldsymbol{h}\mid\boldsymbol{v})\mid H]\mathrm{d}\boldsymbol{v}\\
 &  & =-\mathrm{H}[p(\boldsymbol{h}\mid\boldsymbol{v})\mid(H,\,\boldsymbol{\sigma})]
\end{eqnarray*}
Then $F_{1}$ and $F_{2}$ are completely determined by $\boldsymbol{\sigma}$,
such that their minima do not depend on the choice of $P_{\boldsymbol{\eta}}$
and $Q_{\boldsymbol{\eta}}$. Therefore:
\begin{eqnarray}
 &  & \arg\min_{\boldsymbol{\eta}}D_{\mathrm{KL}}[P_{\boldsymbol{\eta}}\parallel Q_{\boldsymbol{\eta}}]\label{eq:thm:4.1:4}\\
 &  & =\arg\min_{\boldsymbol{\eta}}F_{3}(P_{\boldsymbol{\eta}},\,Q_{\boldsymbol{\eta}})\nonumber \\
 &  & =\arg\max_{\boldsymbol{\eta}}\int_{H,\boldsymbol{\sigma}}p_{\boldsymbol{\eta}}(\boldsymbol{v})p_{\boldsymbol{\eta}}(\boldsymbol{h}\mid\boldsymbol{v})\log q_{\boldsymbol{\eta}}(\boldsymbol{v},\,\boldsymbol{h})\mathrm{d}\boldsymbol{v}\mathrm{d}\boldsymbol{h}\nonumber \\
 &  & =\arg\max_{\boldsymbol{\eta}}\int_{X}p_{\boldsymbol{\eta}}(x)\log q_{\boldsymbol{\eta}}(x)\mathrm{d}\mu(x)\nonumber 
\end{eqnarray}
Since $p_{\boldsymbol{\eta}}(x)$ are the empirical probabilities
over $(X,\,\Sigma)$ the conditions of lemma , equation $\ref{eq:lem:4.3:1}$,
are satisfied, such that
\begin{eqnarray*}
 &  & \arg\max_{\boldsymbol{\eta}}\mathrm{L}[Q_{\boldsymbol{\eta}}\mid\boldsymbol{\sigma}]\\
 &  & \stackrel{}{=}\arg\max_{\boldsymbol{\eta}}\int_{X}p_{\boldsymbol{\eta}}(x)\log q_{\boldsymbol{\eta}}(x)\mathrm{d}\mu(x)\\
 &  & \stackrel{\ref{eq:thm:4.1:4}}{=}\arg\min_{\boldsymbol{\eta}}D_{\mathrm{KL}}[P_{\boldsymbol{\eta}}\parallel Q_{\boldsymbol{\eta}}]\\
 &  & \stackrel{\ref{eq:thm:4.1:3}}{=}\arg\min_{\boldsymbol{\eta}}D_{\psi}[Q_{\boldsymbol{\eta}}\parallel P_{\boldsymbol{\eta}}]\\
 &  & \stackrel{\ref{eq:thm:4.1:2}}{=}\arg\min_{\boldsymbol{\eta}}d(P_{\boldsymbol{\eta}},\,Q_{\boldsymbol{\eta}})
\end{eqnarray*}
\end{proof}

\section{Alternating minimization}

Let $(X,\,\mathcal{P},\,D_{\psi})$ be a dually flat statistical manifold
and $(X,\,\mathcal{Q})$ and $(X,\,\mathcal{E})$ smooth submanifolds
of $(X,\,\mathcal{P},\,D_{\psi})$. Then the divergence $D_{\psi}$
between $(X,\,\mathcal{Q})$ to $(X,\,\mathcal{E})$ is defined by
the minimal divergence of its respective points, such that:
\[
D_{\psi}[\mathcal{Q}\parallel\mathcal{S}]=\min_{Q\in\mathcal{Q},\,S\in\mathcal{S}}D_{\psi}[Q\parallel S]=D_{\psi}[\hat{Q}\parallel\hat{S}]
\]
where $\hat{Q}\in\mathcal{Q}$ and $\hat{S}\in\mathcal{S}$ minimize
$D_{\psi}[Q\parallel S]$. By applying Amari's projection
theorem it follows, that the pair $(\hat{Q},\,\hat{S})$ also minimizes
the geodesic distance and therefore has to be regarded as a pair of
closest points between $(X,\,\mathcal{Q})$ and $(X,\,\mathcal{S})$.
Furthermore the dually flat structure allows an iterative alternating
geodesic projection between $(X,\,\mathcal{Q})$ and $(X,\,\mathcal{S})$
to asymptotically approximate $\hat{Q}$ and $\hat{S}$. This is termed
alternating minimization.
\begin{defn*}[\emph{Alternating minimization}]
\label{def:Alternating-minimization-algorithm} Let $(X,\,\mathcal{U})$
be an exponential family with smooth submanifolds $(X,\,\mathcal{Q})$
and $(X,\,\mathcal{S})$. Then the alternating minimization from $\mathcal{S}$
to $\mathcal{Q}$ iteratively defines a sequence $(Q_{n},\,S_{n})_{n\in\mathbb{N}_{0}}$
of elements in $(\mathcal{Q},\,\mathcal{S})$, which is given by\emph{:}

\textbf{Begin}:

\emph{Let $S_{0}\in\mathcal{S}$ be arbitrary}

\textbf{Iteration}\emph{:}

($m$-step)\emph{ $Q_{n}$ is given by a geodesic projection of $S_{n}$
to $\mathcal{Q}$:
\[
Q_{n}=\pi(Q_{n})=\arg\min_{Q\in\mathcal{Q}}d(S_{n},\,Q)
\]
}

($e$-step)\emph{ $S_{n+1}$ is given by a dual geodesic projection
of $Q_{n}$ to $\mathcal{S}$:
\[
S_{n+1}=\pi^{*}(Q_{n})=\arg\min_{S\in\mathcal{S}}d(S,\,Q_{n})
\]
}
\end{defn*}

\begin{thm}[\emph{Alternating minimization}]
\emph{\label{thm:Alternating-minimization}} Let $(X,\,\mathcal{U})$
be an exponential family model with smooth submanifolds $(X,\,\mathcal{Q})$
and $(X,\,\mathcal{S})$. Then the alternating minimization algorithm
converges in a pair of probability distributions, that locally minimize
the geodesic distance between $(X,\,\mathcal{Q})$ and $(X,\,\mathcal{S})$.
\end{thm}

\begin{figure}[h]
\begin{centering}
\def\svgwidth{\columnwidth} 
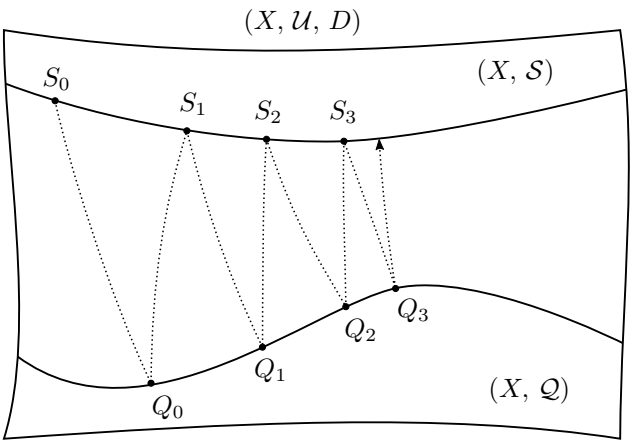
\par\end{centering}
\caption{Alternating minimization in Exponential Families}
\end{figure}

\begin{proof}
Let $(X,\,\mathcal{U})$ be an exponential family, then the Bregman
divergence of its cumulative generating function $\psi$ induces a
Riemannian metric and $(X,\,\mathcal{U},\,D_{\psi})$ is a simply
connected dually flat statistical manifold. Then $(X,\,\mathcal{Q},\,D_{\psi})$
and $(X,\,\mathcal{S},\,D_{\psi})$ are Riemannian submanifolds of
$(X,\,\mathcal{U},\,D_{\psi})$, with respect to the induced Riemannian
metric. In the $m$-step, the geodesic projection $\pi:\mathcal{S}\to\mathcal{Q}$
is given by the dual affine projection in the $m$-parametrisation.
Then the dual affine projection minimizes the Bregman divergence and
the Pythagorean theorem yields:
\begin{align}
 & d(Q_{n},\,S_{n})\label{eq:thm:4.3:1}\\
 & =D_{\psi}[S_{n}\parallel\pi(S_{n})]\nonumber \\
 & \leq D_{\psi}[S_{n}\parallel\pi(S_{n})]+D_{\psi}[\pi(S_{n})\parallel R]\nonumber \\
 & =D_{\psi}[S_{n}\parallel Q_{n-1}]=d(Q_{n-1},\,S_{n})\nonumber 
\end{align}
In the $e$-step dual geodesic projection $\pi^{*}:\mathcal{Q}\to\mathcal{S}$
is given by the affine projection in the $e$-parametrisation. Then
the affine projection minimizes the dual Bregman divergence and the
Pythagorean theorem yields:
\begin{align}
 & d(Q_{n},\,S_{n+1})\label{eq:thm:4.3:2}\\
 & =D_{\psi^{*}}[\pi(Q_{n})\parallel Q_{n}]\nonumber \\
 & \leq D_{\psi^{*}}[R\parallel\pi(Q_{n})]+D_{\psi^{*}}[\pi(Q_{n})\parallel Q_{n}]\nonumber \\
 & =D_{\psi^{*}}[S_{n}\parallel Q_{n}]=d(Q_{n},\,S_{n})\nonumber 
\end{align}
This proves, that $d(Q_{n},\,S_{n})$ monotonously decreases, since:

\begin{align*}
 & d(Q_{n+1},\,S_{n+1})\\
 & \stackrel{\ref{eq:thm:4.3:2}}{\leq}d(Q_{n},\,S_{n+1})\\
 & \stackrel{\ref{eq:thm:4.3:1}}{\leq}d(Q_{n},\,S_{n}),\,\forall n\in\mathbb{N}_{0}
\end{align*}
Furthermore $d(Q_{n},\,S_{n})$ is bounded bellow by:
\[
d(Q_{n},\,S_{n})=D_{\psi}[S_{n}\parallel S_{n}]\geq0,\,\forall n\in\mathbb{N}_{0}
\]
This proves, that $d(Q_{n},\,S_{n})$ converges against a local minimum.
\end{proof}
\begin{cor}
\label{cor:4.2}Let $(X,\,\mathcal{U})$ be an exponential family
with an $e$-flat submanifold $(X,\,\mathcal{Q})$ and an $m$-flat
submanifold $(X,\,\mathcal{S})$. Then the alternating minimization
algorithm from $\mathcal{S}$ to $\mathcal{Q}$ converges against
points, that globally minimize the geodesic distance between $(X,\,\mathcal{Q})$
and $(X,\,\mathcal{S})$.
\end{cor}

\begin{proof}
Let\emph{ }$(Q_{n},\,S_{n})_{n\in\mathbb{N}_{0}}$ be a sequence,
given by the alternating minimization algorithm from $\mathcal{S}$
to $\mathcal{Q}$, then due to theorem \ref{thm:Alternating-minimization}
$(Q_{n},\,S_{n})$ converges against against points, that locally
minimize the geodesic distance between $(X,\,\mathcal{Q})$ and $(X,\,\mathcal{S})$.
Since $(X,\,\mathcal{Q})$ is $e$-flat and $(X,\,\mathcal{S})$ is
$m$-flat corollary \ref{cor:3.2} is satisfied, such that this geodesic
distance in unique and therefore $(Q_{n},\,S_{n})$ converges against
points, that globally minimize the geodesic distance between $(X,\,\mathcal{Q})$
and $(X,\,\mathcal{S})$.
\end{proof}
\begin{cor}
\label{cor:4.3}Let $(X,\,\mathcal{P})$ be an exponential family
over a partially observable measurable space $(X,\,\Sigma)$ and $(X,\,\mathcal{E})$
an absolutely continuous empirical model over $(X,\,\Sigma)$. Let
further be $(X,\,\mathcal{Q})$ an $e$-flat submanifold of $(X,\,\mathcal{P})$,
and $\boldsymbol{\sigma}$ a repeated observation over $(X,\,\Sigma)$.
Then a maximum likelihood estimation of $(X,\,\mathcal{Q})$ respective
to $\boldsymbol{\sigma}$ is given by the limit of alternating minimization
algorithm from $\mathcal{E}_{\boldsymbol{\sigma}}$ to $\mathcal{Q}$.
\end{cor}

\begin{proof}
Since $(X,\,\mathcal{P})$ is an exponential family and $(X,\,\mathcal{E})$
an absolutely continuous empirical model, there exists a common dually
flat embedding space $(X,\,\mathcal{U}_{\boldsymbol{\eta}},\,D_{\psi})$,
such that $(X,\,\mathcal{P})$ and $(X,\,\mathcal{E})$ are dually
flat submanifolds, with respect to the Bregman divergence $D_{\psi}$.
Then $(X,\,\mathcal{E}_{\boldsymbol{\sigma}})$ is $m$-flat in $(X,\,\mathcal{U}_{\boldsymbol{\eta}},\,D_{\psi})$
and by definition $(X,\,\mathcal{Q})$ is $e$-flat in $(X,\,\mathcal{P})$
and therefore also in $(X,\,\mathcal{U}_{\boldsymbol{\eta}},\,D_{\psi})$.
This allows the application of corollary \ref{cor:4.2}.
\end{proof}

\bibliographystyle{unsrt}
\bibliography{articles}

\end{document}